\documentclass[11pt]{article}
\usepackage[utf8]{inputenc}
\usepackage[multiple]{footmisc}
\usepackage{amsmath}               % great math stu
\usepackage{amsfonts}              % for blackboard bold, etc
\usepackage{amssymb}
\usepackage{amsthm}               % better theorem environments
\usepackage{graphicx,subfigure}

\usepackage{bigints}

\usepackage[disable]{todonotes}
\usepackage{fullpage}

\theoremstyle{plain}

\newtheorem{lemma}{Lemma}[section]
\newtheorem{theorem}[lemma]{Theorem}

\theoremstyle{definition}

\newtheorem{remark}[lemma]{Remark}
\newtheorem{definition}[lemma]{Definition}

\theoremstyle{remark}

\newcommand{\R}{{\mathbb R}}

\renewcommand{\Re}{\mbox{\rm Re}\,}
\renewcommand{\Im}{\mbox{\rm Im}\,}
\date{\today}

\title{Orbital stability of standing waves for the nonlinear Schr{\"o}dinger equation with attractive delta potential and double power repulsive nonlinearity}

\author{Jaime Angulo Pava\thanks{Department of Mathematics,
IME-USP, Rua do Mat\~ao 1010, Cidade Universit\'aria, CEP 05508-090,
 S\~ao Paulo, SP (Brazil). E-mail: \texttt{angulo@ime.usp.br}}\and C\'{e}sar A. Hern\'{a}ndez Melo \thanks{Department of Mathematics, DMA-UEM, Av. Colombo, 5790 Jd. Universit\'ario, CEP 87020-900, Maring\'a, PR (Brazil). E-mail: \texttt{cahmelo@uem.br}}  %\thanks{Corresponding author: \texttt{cahmelo@uem.br}; Tel.: +55 44-3011-5358; Fax: +55 11-3011-3873.} 
\and Ram\'on G. Plaza\thanks{Instituto de Investigaciones en Matem\'aticas Aplicadas y en Sistemas, Universidad Nacional Aut\'{o}noma de M\'{e}xico, Circuito Escolar s/n, Ciudad Universitaria, C.P. 04510 Cd. de M\'{e}xico (M\'exico). E-mail: \texttt{plaza@mym.iimas.unam.mx}} }

\begin{document}

\maketitle

%\setuptodonotes[fancyline, color=yellow!40]

\begin{abstract}
In this paper, a nonlinear Schr\"odinger equation with an attractive (focusing) delta potential and a repulsive (defocusing) double power nonlinearity in one spatial dimension is considered. It is shown, via explicit construction, that both standing wave and equilibrium solutions do exist for certain parameter regimes. In addition, it is proved that both types of wave solutions are orbitally stable under the flow of the equation by minimizing the charge/energy functional.
\end{abstract}

%%%%\nocite{*}

\section{Introduction}
\label{int}

This work addresses the orbital stability of peak-standing waves associated to the following double power nonlinear Schr\"odinger (NLS) equation with a point interaction determined by Dirac $\delta$ distribution centered at the origin (henceforth NLSDP),
\begin{equation}\label{deltasch1}
iu_{t}+u_{xx}+Z\delta(x)u+\lambda_1u|u|^{p-1}+\lambda_2 u|u|^{2p-2}=0,\quad\text{for}\;\;t, x\in \mathbb{R},
\end{equation}
here $u=u(x,t)\in \mathbb C$, $1 < p < \infty$, $\lambda_1\leq 0,$ $\lambda_2< 0$ and  $Z\in \mathbb{R}$ represents the so-called strength parameter.

\todo[fancyline, color=yellow!40, size=\small]{\textbf{(B)}Write more about the physical model. Trace back references with care...}
Equation \eqref{deltasch1} belongs to a family of models featuring the competition between  repulsive ($\lambda_1\leq 0$)/attractive ($\lambda_1\geq 0$)  and repulsive ($\lambda_2\leq 0$)/attractive ($\lambda_2\geq 0$)/terms, that have drawn considerable attention for $p=3$ in both the physical and the mathematical communities in recent years (an abridged list of references include \cite{AdNo13}, \cite{AdNoV13}, \cite{Agraw5ed}, \cite{Ang12}, \cite{AnAr18}, \cite{AnGo17}, \cite{AnGo18},  \cite{AnPo13}, \cite{BKo04}, \cite{CaMiR05}, \cite{DMADDKK95}, \cite{FuJe08}, \cite{FOO08}, \cite{GeMW16},  \cite{GiDriMa04},  \cite{GHW04},  \cite{LFFKS08}, \cite{Men93}, \cite{MoNe04}, \cite{PKH11}, \cite{SaTa04} and \cite{SeCaHo05}). This combination of nonlinearities in \eqref{deltasch1} for $p>1$ is well-known in optical media (cf. \cite{BCLTSS03}, \cite{FABLS13}, \cite{FAR07}). In particular and in the context of nonlinear optics, we recall that for a effective linear
potential term, $V(x)$, the general NLS model
\[
iu_t+ u_{xx} + V(x)u +F(u)=0,
\]
represents a trapping (wave-guiding) structure for light beams induced by
an inhomogeneity of the local refractive index. In particular, the delta-function term in \eqref{deltasch1} adequately represents a narrow trap which is able to capture broad solitonic beams (see \cite{CaMiR05,GHW04}). In the description of Bose-Einstein condensates \cite{DMADDKK95,PiSt03}, the same equation (also known as the Gross-Pitaevskii equation) with a delta potential models the dynamics of a condensate in the presence of an impurity of a small length scale (cf. \cite{Kon05,SeCaHo05}). Notably, in both physical theories the most common forms of the nonlinearity $F(u)$ are either a single cubic (usually attractive) term, or the focusing cubic plus a defocusing quintic nonlinearity, modeling the interplay of multi-body boson interactions in the case of Bose-Einstein condensates (cf. \cite{Kon05}), or a combination of channel waveguides and saturable nonlinearities (the simplest among which is the cubic-quintic) in the case of nonlinear optics (see, e.g., \cite{GiDriMa04}). 

The existence of special solutions of equation \eqref{deltasch1} called ``standing wave" solutions of the form
\begin{equation}\label{stand}
u(x,t)=e^{-i\omega t} \phi_\omega(x),\;\;\;\;\;\;\;\;\omega\in I\subset \mathbb{R},
\end{equation}
where the profile $\phi_\omega:\mathbb R\to \mathbb R$ of the wave satisfies in a distributional sense the elliptic equation
\begin{equation}\label{peak}
\Big(-\frac{d^2}{dx^2}-Z\delta(x)\Big)\phi_\omega-\omega \phi_\omega-\lambda_1 \phi_\omega|\phi_\omega|^{p-1}-\lambda_2 \phi_\omega|\phi_\omega|^{2p-2}=0,
\end{equation}
 with $\phi_\omega$ belonging to the domain of formal $\delta$-interaction quantum operator $A_Z=-\partial_{xx}-Z\delta(x)$, defined as 
 \begin{equation}\label{A_Z}
\left\{
\begin{aligned}
A_Zf(x)&=-f''(x)\qquad x\neq 0,\\
D(A_Z)&=\{f\in H^1(\mathbb R)\cap H^2(\mathbb R \backslash \{0\}): f'(0+)-f'(0-)=-Zf(0)\},
\end{aligned}
\right.
\end{equation}
has been widely considered in analytic, numerical and experimental works for specific values of the parameters $Z, \omega, \lambda_1,  \lambda_2$. Indeed,  for $Z=0$ (that is, in the case of no point defects), the rigorous existence and stability analysis of standing waves for NLSDP model with general double-power nonlinearities have been studied in the works by Maeda \cite{Mae08} and Ohta \cite{Oht95}. 
 
 The existence and stability of standing waves for the model with $Z \neq 0$, $\lambda_1\neq 0$, $\lambda_2=0$, $\omega<0$, $p>1$, have been extensively discussed earlier by Fukuizumi and Jeanjean \cite{FuJe08} in the case of a repulsive delta potential; by Fukuizumi \emph{et al.} \cite{FOO08} in the attractive case; and by Le Coz \emph{et al.} \cite{LFFKS08} in the case of a repulsive defect and subcritical nonlinearity.  For this specific choice of parameters, but in a periodic framework, the reader is referred to the works of Angulo \cite{Ang12} and Angulo and Ponce \cite{AnPo13}. The case $Z\neq 0$, $p=3$ (i.e., the cubic/quintic model), $\lambda_1>0$ (attractive cubic interaction) and $\lambda_2\in\mathbb{R}\setminus\{0\}$ has been recently analyzed by the first two authors in \cite{AnHe19}. The case of an attractive delta potential with attractive cubic with repulsive quintic interactions has been recently studied by Genoud \emph{et al.} \cite{GeMW16}.

Now, it is well-known (see Lemmata \ref {nule} and \ref{cinco} below) that for arbitrary values of parameters $\lambda_1$ and  $\lambda_2$,  NSLDP model may not have standing wave solutions vanishing at infinity (still in the case $Z\neq 0$). Moreover, it may happen that exact solutions are not  available in general. Recently, Kaminaga and Ohta \cite{KaOh09} studied the stability of a family of explicit standing wave solutions for the NLSDP model with a simple power repulsive ($\lambda_1=-1, \lambda_2=0$) nonlinearity, with a focusing $\delta$-interaction ($Z>0$) and with a wave phase velocity $-\omega$ satisfying $0<-\omega<\frac{Z^2}{4}$.
 
Up to our knowledge, the existence and stability of standing waves for the NLSDP model with a double power repulsive nonlinearity has not been considered in the current literature. It is to be noticed that multi-body interactions of the same sign appear in the study of Bose-Einstein condensates (see, for example, Belobo \emph{et al.} \cite{BeBeKo14}, Brazhnyi and Konotop \cite{BKo04}, Kamchatnov and Salerno \cite{KamSa09} and Kamchatnov and Korneev \cite{KaKo10}). Therefore, the main focus of this work is to study the existence and  stability of  standing wave solutions for NLSDP model when $p>1$, with a double power repulsive or defocusing ($\lambda_1\leq 0, \lambda_2<0$) nonlinearity, with a focusing (attractive) $\delta$-interaction ($Z>0$) and with a wave phase velocity $-\omega$ satisfying 
\[
-\frac{p\lambda_1^2}{(p+1)^2\lambda_2}<-\omega<\frac{Z^2}{4}.
\]
In fact, for 
\[
\alpha=\frac{\lambda_1}{p+1},\hspace{0.6cm}\beta=\frac{\lambda_2}{p},\hspace{0.5cm}\text{and}\hspace{0.5cm}l=\frac{1}{(p-1)\sqrt{-\omega}}\sinh^{-1}\left(\frac{\alpha}{\sqrt{\omega\beta-\alpha^2}}\right),
\] 
it will be shown that the functions 
\begin{equation}\label{numeratorwith}
\phi_{\omega}(x)=\left[\frac{\alpha}{-\omega}+\frac{\sqrt{\omega\beta -\alpha^2}}{-\omega}\sinh\left((p-1)\sqrt{-\omega}\left(|x|+R_1^{-1}\left(\frac{Z}{2\sqrt{-\omega}}\right)\right)\right)\right]^{-\frac{1}{p-1}},
\end{equation}
where $R_1:(-l,\infty)\rightarrow (1,\infty)$ is the diffeomorphism  given by
\begin{equation}
\label{Er}
R_1(d)=\frac{\sqrt{\omega\beta-\alpha^2}\cosh((p-1)\sqrt{-\omega}d)}{\sqrt{\omega\beta-\alpha^2}\sinh((p-1)\sqrt{-\omega}d)+\alpha},
\end{equation}
constitute a family of standing wave profiles (solutions of the equation \eqref{peak}) for the parameter values given above. 

Moreover, we show that the standing wave solutions determined by the profile $\phi_{\omega}$ in \eqref{numeratorwith} for $\lambda_1\leqq 0, \lambda_2 <0$, are orbitally stable in $H^1(\mathbb{R)}$ (see Theorem \ref{main} below). We note that since the classical translation symmetry associated to the equation \eqref{deltasch1} when $Z=0$ does not hold in the case $Z\neq 0$, then our concept of orbital stability pertains to the phase-invariance symmetry associated to equation \eqref{deltasch1}. More precisely we have the following

\begin{definition}\label{dsta} The standing wave $e^{-i\omega t}\phi_{\omega}$ is \emph{orbitally stable} by the
flow of the NLSDP model \eqref{deltasch1} in $H^1(\mathbb{R})$, if for any $\epsilon>0$ there exists $\delta>0$ such that if $\| u_0-\phi_{\omega}\|_{H^1} <\delta$ then
\[
\inf_{\theta\in\mathbb{R}} \|u(t)-e^{i\theta}\phi_{\omega}\|_{H^1}<\epsilon, \qquad \text{for all } \, t \in \mathbb R,
\]
where $u(t)$ denotes the solution to equation  \eqref{deltasch1} with initial data $u(0)=u_0\in H^{1}(\mathbb{R})$. Otherwise, $e^{-i\omega t}\phi_{\omega}$ is said to be \emph{orbitally unstable} in $H^{1}(\mathbb{R})$.
\end{definition}

Thus, our main stability result for the peak standing waves profiles in \eqref{numeratorwith} is the following.

\begin{theorem}\label{main} 
Let $1 < p < \infty$, $\lambda_1 \leqq 0$, $\lambda_2 <0$  and  $Z>0$ in equation \eqref{deltasch1} be such that $-\frac{p\lambda_1^2}{(p+1)^2\lambda_2} < \frac{Z^2}{4}$. Then for all values of $\omega<0$ satisfying 
\[
-\frac{p\lambda_1^2}{(p+1)^2\lambda_2}<-\omega<\frac{Z^2}{4}
\]
and for the profile $\phi_{\omega}$ defined in \eqref{numeratorwith}, the standing wave solution  $e^{-i\omega t}\phi_{\omega}$ is orbitally stable by the flow of the NLSDP model \eqref{deltasch1} in $H^1(\mathbb{R})$.
\end{theorem}
\begin{remark}
(a) Theorem \ref{main} generalizes the stability result established by Kaminaga and Ohta (see Theorem 1 in  \cite{KaOh09}). Indeed, if we take $p=\frac{r+1}{2}>1$, $\lambda_1=0$, $\lambda_2=-1$ in the equation \eqref{deltasch1} and formula \eqref{numeratorwith} then we deduce that
the Schr\" odinger equation
\[
iu_t+u_{xx}+Z\delta(x)u-|u|^{r-1}u=0,
\]   
has stable standing wave solutions given by 
\begin{equation}\label{solKamiMas}
u(x,t)=e^{-i\omega t}\left[\frac{-\omega(r+1)}{2}\right]^{\frac{1}{r-1}}\left[\sinh\left(\frac{(r-1)\sqrt{-\omega}}{2}|x|+\tanh^{-1}\left(\frac{2\sqrt{-\omega}}{Z}\right)\right)\right]^{-\frac{2}{r-1}},
\end{equation}
providing  the parameters $\omega, Z$ satisfy $Z>0$ and $0<-\omega<Z^2/4$.\\

\noindent (b) Lemma \ref{nule} below shows that in the case of parameter values $Z>0$ and $ \omega+ Z^2/4\leq 0$ there are no non-trivial solutions to \eqref{peak}. If $Z\in \mathbb R\setminus \{0\}$ and $ \omega> 0$ then there are no non-trivial solutions to \eqref{peak} either (see Lemma \ref{cinco} below). For $Z<0$, from Lemma \ref{Z<0} below, we are no non-trivial solutions to \eqref{peak} in the cases of  $\lambda_1, \lambda_2<0$.
\end{remark}

Our second focus of attention is the existence and stability of equilibrium solutions for NLSDP of the form
\[
u(x,t) = \phi_0(x),
\]
which are particular case of standing waves \eqref{stand} with phase velocity  $\omega = 0$. In fact, we establish that, for $1 < p < 5$, $Z > 0$ and double power repulsive nonlinearities with $\lambda_1 < 0$ and $\lambda_2 < 0$, the $L^2(\mathbb R)$-rational profile 
\begin{equation}
\label{phi0}
\phi_0(x)=\left[\frac{-2p(p+1)\lambda_1}{p(p-1)^2\lambda_1^2\left(|x|+R_2^{-1}\left(\frac{Z}{4}\right)\right)^2+(p+1)^2\lambda_2} \right]^{\frac{1}{p-1}},
\end{equation}
is an equilibrium solution to equation \eqref{deltasch1} with $\omega = 0$ and for the parameter values under consideration.  Here $R_2 : (-l_0, \infty) \to (0, \infty)$ is the diffeomorphism defined by
\[
R_2(d)=\frac{d}{(p-1)(d^2 - l_0^2)},
\]
with
\[
l_0 :=\frac{\sqrt{|\lambda_2|}(p+1)}{\sqrt{p}(p-1)\lambda_1} < 0.
\]

Our stability result associated to the profiles in \eqref{phi0}  is the following

\begin{theorem}
\label{main2}
Let $1 < p < 5$, $\lambda_1 < 0$, $\lambda_2 < 0$ and $Z > 0$. Then the equilibrium solution $u(x,t) = \phi_0(x)$ to the NLSDP model \eqref{deltasch1}, with $\omega = 0$ in \eqref{peak} and  $\phi_0$ defined in \eqref{phi0}, is orbitally stable in $H^1(\R)$.
\end{theorem}

Next, we describe our strategy for proving Theorems \ref{main} and \ref{main2}. Our approach is of variational type and based on the general concentration-compactness method introduced by Cazenave and Lions \cite{CaLi82} (see also \cite{KaOh09}). In the case where $\omega \leq 0$, the principal functional is the following charge/energy conserved action 
\begin{equation}\label{ge}
G_{\omega}(v)=\frac{1}{2}\|v_x \|^2_{L^2}-\frac{Z}{2}|v(0)|^{2}-\frac{\omega}{2}\|v\|^2_{L^2}-\frac{\lambda_1}{p+1}\|v\|^{p+1}_{L^{p+1}}-\frac{\lambda_2}{2p}\|v\|^{2p}_{L^{2p}}.
\end{equation}
Our analysis of existence and stability of standing wave is divided into the following steps:
\begin{enumerate}
\item[(i)]\label{cp} {\it The Cauchy problem}: The initial value problem associated to the NLSDP equation \eqref{deltasch1} is globally well-posed in $H^1(\mathbb R)$ for $1<p<+\infty$, $\lambda_1\leq 0$, $\lambda_2<0$ and $Z>0$.

\item[(ii)]\label{sp} {\it The stationary problem}: The set ${\cal A}_{\omega}$ of non-trivial solutions of equation \eqref{peak} in $H^1(\mathbb{R})$ will be  characterized, via uniqueness, by
\begin{equation}\label{Stp}
{\cal A}_{\omega}=\{v:G'_{\omega}(v)=0, v\neq0\}=\{e^{i\theta}\phi_{\omega}: \theta\in\mathbb{R}\}.
\end{equation}    
Here, the parameters $p, \lambda_1, \lambda_2, Z$ and  $\omega$ appearing in equation \eqref{peak} satisfy the assumptions of Theorem \ref{main}. The function $\phi_{\omega}$ denotes the standing wave profile given in \eqref{numeratorwith}.
\item[(iii)]\label{mp} {\it The minimization problem}: For $p, \lambda_1, \lambda_2, Z$ and  $\omega$ satisfying  the assumptions of Theorem \ref{main}, the  quantity 
\[
m(\omega)=\inf \{G_{\omega}(v):v\in H^1(\mathbb{R})\},
\] 
satisfies the following properties:
\begin{enumerate}
\item $-\infty<m(\omega)<0$; and,
\item any sequence $h_n\in H^1(\mathbb{R})$ such that $\lim_{n\to\infty}G_{\omega}(h_n)=m(\omega)$ admits a subsequence converging to some $h\in H^{1}(\mathbb{R})$ with $G_{\omega}(h)=m(\omega)$.
\end{enumerate}
\end{enumerate}
Property (ii) implies that the solution to equation \eqref{peak} is unique modulo phase-invariance symmetries and the sign of the profile (see Lemma \ref{seis} below). If properties (a) and (b) in (iii) are satisfied  then is possible to prove that the set of minimizers, $M(\omega)=\{u\in H^1(\mathbb{R}):G_{\omega}(u)=m(\omega)\}$, satisfies
\[
M(\omega)={\cal A }_{\omega},
\]
that is, $M(\omega)$ coincides with the set of non-trivial critical points of the functional $G_{\omega}$ in \eqref{ge}, as well as with the orbit generated by the standing wave profile $\phi_{\omega}$. This allows us, in turn, to prove the stability  Theorem \ref{main} (see section \ref{secstab} below). The proof of Theorem \ref{main2}, case $\omega=0$, follows the same guidelines.
 
\subsection*{Plan of the paper} This paper is organized as follows. Section \ref{secexist} is devoted to establish local and global well-posedness of the Cauchy problem for the NLSDP model \eqref{deltasch1}. Section \ref{secstanding} contains the general construction of the profile $\phi_{\omega}$ in \eqref{numeratorwith} for $\lambda_2<0$ and $\lambda_1 \leq 0$, as well as the construction of the explicit equilibrium solution $\phi_0$ in \eqref{phi0} for $\lambda_1 <0$, $\lambda_2 < 0$. Section \ref{secstab} describes the set of non-trivial critical points of the charge/energy functional in a general setting and contains the proof of uniqueness of solutions for \eqref{peak} (modulo rotations). It also contains the proof of orbital stability of the standing waves \eqref{numeratorwith} with $\omega \neq 0$ (Theorem \ref{main}), as well as of that for rational-equilibrium solutions \eqref{phi0} (Theorem \ref{main2}).
 
 \subsection*{Notation} By $\|\cdot\|_{L^p}$ we denote the norm in $L^p(\R)$, except where it is explicitly stated otherwise. The inner product in  $L^2(\mathbb R)$  is defined by   $\langle f,g\rangle= \Re \int_{-\infty}^{+\infty} f(x)\overline{g(x)}dx$. According to custom, standard Sobolev spaces are denoted by $H^k(\R),\,\,k\in \mathbb N$.
 
\section[Local and global well posedness]{Local and global well-posedness of the NLSDP model}
\label{secexist}

In this section we establish the local and global well-posedness of the Cauchy problem associated to the NLSDP equation in $H^1(\mathbb{R})$, namely
\begin{equation}\label{cachy12}
\left\{
\begin{array}{lll}
\displaystyle iu_{t}-A_Zu+u(\lambda_1 |u|^{p-1}+\lambda_2|u|^{2p-2})=0,\\
u(0)=u_0\in H^1(\mathbb{R}),
\end{array}
\right.
\end{equation}
where $A_Z$ represents the formal $\delta$-point interaction operator, $A_Z:=-\frac{d^2}{dx^2}-Z\delta(x)$, which is defined in \eqref{A_Z}. We recall that this formal expression represents  all the self-adjoint extensions associated to  the following closed, symmetric, densely defined linear operator (see \cite{AGHH2ed}): 
\begin{equation*}
\left\{
\begin{aligned}
A_0&=-\frac{d^2}{dx^2}\\
D(A_0)&=\{g\in H^2(\mathbb{R}): g(0)=0 \}.
\end{aligned}
\right.
\end{equation*}

Upon application of the  First Representation Form Theorem (cf. Kato \cite{Kato80}, chapter 6), it is possible to show that the associated form to $A_Z$ is given by
\begin{equation}\label{quadra}
 F_Z[u,v] = \Re\int_{-\infty}^{+\infty} u'(x)\overline{v'(x)}dx - Z \, \Re (u(0)\overline{v(0)}),
\end{equation}
where $(u,v)\in D( F_Z)=H^1(\mathbb R)\times H^1(\mathbb R)$. The bilinear  form defined above is closed and bounded below. In addition, operator $A_Z=-\frac{d^2}{dx^2}-Z\delta(x)$ can be extended as a linear bounded operator $u\to A_Z u$ from $H^1(\mathbb R)$ to $H^{-1}(\mathbb R)$. Indeed, this action is defined by
\begin{equation}\label{quadra2}
\langle A_Z u, v\rangle= F_Z[u,v],\qquad\text{for}\;\;u, v\in H^1(\mathbb R).
\end{equation}

Since our approach is based on the abstract results by Cazenave (see \cite{Caz03}, chapter 3), we first establish the following spectral properties of $A_Z $. Indeed, for $Z\in \mathbb R$ we have that the essential spectrum of $A_Z $,  $\Sigma_{\mathrm{ess}}(A_Z)$, is the nonnegative real axis, $\Sigma_{\mathrm{ess}}(A_Z)=[0,+\infty)$. For $Z>0$, $A_Z $ has exactly one negative, simple eigenvalue, i.e., its discrete spectrum, $\Sigma_{\mathrm{dis}}(A_Z )$, is $\Sigma_{\mathrm{dis}}(A_Z )=\{{-Z^2/4}\}$, with a strictly (normalized) eigenfunction $
 \Psi_Z(x)=\sqrt{\frac{Z}{2}}e^{-\frac{Z}{2}|x|}$. Thus,
\[
 \inf\Big\{\|v_x\|_{L^2}^2-Z|v(0)|^2: \|v \|_{L^2}=1, Z>0\Big \}=-\frac{Z^2}{4}.
\]
 
For $Z\leq 0$,  $A_Z $ has not discrete spectrum, $\Sigma_{\mathrm{dis}}(A_Z )=\varnothing$. Therefore the
operators $A_Z$ are bounded from below, more precisely,
\begin{equation}\label{boundbelo}
\begin{cases}
A_Z \geq -Z^2/4, & Z>0,\\A_Z \geq 0, & Z< 0
\end{cases}
\end{equation}
(see \cite{AGHH2ed}, chapter I.3, for further information).

\begin{theorem}[local well-posedness]
\label{cazi} 
For any $u_0\in H^1(\mathbb{R})$ and $Z\in \mathbb R$, there exists $T>0$ and a unique solution  $u\in C([-T,T]; H^1(\mathbb{R}))\cap C^1([-T,T]; H^{-1} (\mathbb{R}))$ to the Cauchy problem \eqref{cachy12} with  $u(0)=u_0$ such that 
\begin{equation}\label{blow}
\lim_{t\to T^{-}}\|u(t)\|_{H^1}=+ \infty,\hspace{0.5cm}\text{if } T<\infty.  
\end{equation}
For each $T_0\in (0,T)$ the mapping $u_0\in H^1(\mathbb{R}) \to u\in C([-T_0,T_0]; H^1(\mathbb{R}))$ 
is continuous. Moreover, the solution $u(t)$ satisfies conservation of charge and energy:
\begin{equation}\label{fi}
\|u(t)\|_{L^2}=\|u_0\|_{L^2}, \qquad E(u(t))=E(u_0),
\end{equation}
for all $t\in[-T,T]$, where the energy functional $E$ is defined as
\begin{equation}\label{Energy}
E(v) := \frac{1}{2}\|v_x \|^2_{L^2}-\frac{Z}{2}|v(0)|^{2}-\frac{\lambda_1}{p+1}\|v\|^{p+1}_{L^{p+1}}
-\frac{\lambda_2}{2p}\|v\|^{2p}_{L^{2p}}, \qquad v\in H^1(\mathbb{R}).
\end{equation}
\end{theorem}
\begin{proof}
The proof of this theorem is a direct application of Theorem 3.7.1 in \cite{Caz03}. In fact, from \eqref{boundbelo}, we have the self-adjoint operator ${\cal A}\equiv -A_{Z}-\beta$ on the space $X=L^2(\mathbb{R})$, with $\beta =Z^2/4$ for $Z>0$ and  $\beta =0$ for $Z\leq 0$, and domain $D({\cal A})=D(A_{Z})$, satisfies ${\cal A}\leq 0$. Furthermore, in the present case let us consider the space $X_{{\cal A}}=H^1(\mathbb{R})$ with norm 
\[
\|u\|^2_{X_{{\cal A}}}=\|u_x\|_{L^2}^2+(\beta +1)\|u\|_{L^2}^2-Z|u(0)|^2,
\] 
which is equivalent to the usual norm in $H^1(\mathbb{R})$. Hence, it is possible to verify uniqueness of the solution and that the conditions (3.7.1), (3.7.3) - (3.7.6) in \cite{Caz03} are satisfied with $r=\rho=2$. Finally, condition (3.7.2) in \cite{Caz03} also holds since ${\cal A}$ is a self-adjoint operator in $L^2(\mathbb{R})$.    
\end{proof}

\begin{remark}
\label{remGagliardoNir}
It is to be observed that, for any $p > 1$, $H^1(\R) \subset L^2(\R) \cap L^{p+1}(\R) \cap L^{2p}(\R)$, inasmuch as the Gagliardo-Nirenberg interpolation inequality (see, e.g., Theorem 12.82 in \cite{Leoni2ed}) yields
\begin{equation}
\label{gagliNir}
\begin{aligned}
\|u\|_{L^{p+1}} &\leq C_1 \|u \|_{L^2}^{\theta_1} \|u_x \|_{L^2}^{1-\theta_1},\\
\|u\|_{L^{2p}} &\leq C_2 \|u \|_{L^2}^{\theta_2} \|u_x \|_{L^2}^{1-\theta_2},
\end{aligned}
\end{equation}
with uniform constants $C_j > 0$ and $\theta_1 = (p+2)/(2p+2) \in (0,1)$, $\theta_2 = (p+1)/2p \in (0,1)$.
\end{remark}

Before showing global well-posedness we need an auxiliary result. Let us define the following $C^1$ functional in $H^1(\R)$,
\begin{equation}
\label{deffuncR}
R(v) := \frac{1}{2}\|v_x \|^2_{L^2}-\frac{Z}{2}|v(0)|^{2} -\frac{\lambda_2}{2p}\|v\|^{2p}_{L^{2p}} = E(v) + \frac{\lambda_1}{p+1}\|v\|^{p+1}_{L^{p+1}}, \qquad v\in H^1(\mathbb{R}).
\end{equation}
\begin{lemma}
\label{lemaux}
Let $1 < p < \infty$, $\lambda_1 \leq 0$, $\lambda_2 < 0$ and $Z > 0$. Then there exists a uniform constant $C = C(p,Z) > 0$ such that
\begin{equation}
\label{starR}
\frac{Z}{2}|v(0)|^2 \leq R(v) + C, \qquad \text{for all} \;\;v \in H^1(\R).
\end{equation} 
\end{lemma}
\begin{proof}
By Sobolev inequality (cf. \cite{Brez11}),
\[
|v(0)| \leq \sqrt{2} \|v\|_{L^2(-1,1)}^{1/2} \|v_x\|_{L^2(-1,1)}^{1/2},
\]
for any $v \in H^1(\R)$. Thus, for any $Z > 0$ there exists $C_1 = C_1(Z) > 0$ such that
\[
Z |v(0)|^2 \leq \frac{1}{2} \|v_x\|_{L^2}^2 + C_1 \|v\|_{L^2(-1,1)}^2.
\]
Apply H\"older's and Young's inequalities to estimate
\[
\|v\|_{L^2(-1,1)}^2 = \int_{-1}^1 |v|^2 \, dx \leq 2^{(p-1)/p} \left( \int_{-1}^1 |v|^{2p} \, dx\right)^{1/p} \leq \delta \|v \|_{L^{2p}(-1,1)}^{2p} + 2C_\delta,
\]
for any $\delta > 0$. Since $\lambda_2 < 0$, choose $\delta = - \lambda_2/(2pC_1) > 0$ to obtain
\[
Z |v(0)|^2 \leq \frac{1}{2} \|v_x\|_{L^2}^2 - \frac{\lambda_2}{2p} \|v \|_{L^{2p}}^{2p} + 2C_1 C_\delta,
\]
yielding \eqref{starR}, as claimed.
\end{proof}

\begin{theorem}[global well-posedness]
\label{gwpdel}
For every $p>1$, $Z\in \mathbb R$, $\lambda_1 \leq 0$ and $\lambda_2 < 0$ the Cauchy problem \eqref{cachy12} is globally well-posed  in $H^1(\mathbb{R})$.
\end{theorem}
\begin{proof} Let $u \in C([-T,T]; H^1(\mathbb{R}))\cap C^1([-T,T]; H^{-1} (\mathbb{R}))$ be the local solution to the Cauchy problem \eqref{cachy12} from Theorem \ref{cazi}. From \eqref{Energy}, we can write the following equality
\begin{equation}\label{Timyr}
\frac{1}{2}\|u_x(t)\|^2_{L^2}=E(u(t))+\frac{Z}{2}|u(t)|^{2}+\frac{\lambda_1}{p+1}\|u(t)\|^{p+1}_{L^{p+1}}+\frac{\lambda_2}{2p}\|u(t)\|^{2p}_{L^{2p}},
\end{equation}
for $t \in (-T,T)$. Then, for $Z > 0$  we get immediately that
\begin{equation}\label{Timyrw}
\frac{1}{2}\|u_x(t)\|^2_{L^2}\leq E(u(t))+\frac{Z}{2}|u(t)|^{2}.
\end{equation}
By Lemma \ref{lemaux} there exist a uniform constant $C > 0$ such that
\begin{equation}\label{clave1}
\frac{Z}{2}|u(t)|^2\leq R(u(t))+C.
\end{equation}
Thus, from \eqref{Timyrw},  \eqref{clave1} and \eqref{deffuncR} we arrive at
\[
\frac{1}{2}\|u_x(t)\|^2_{L^2}\leq E(u(t))+R(u(t))+C\leq2E(u(t))+C.
\]
In view that $u$ conserves charge and energy we finally conclude that
\[
\|u(t)\|^2_{H^1}\leq 4E(u(0))+\|u(0)\|^2_{L^2}+2C,
\]
which implies, together with \eqref{blow}, that the time of existence of the solution $u$ is $T=+\infty$. 

In the case where $Z \leq 0$, we immediately get from \eqref{Timyr} that $\|u_x(t)\|^2_{L^2}\leq 2E(u(0))$. This concludes the proof.
\end{proof}

\section{Standing waves and equilibrium solutions}
\label{secstanding} 

This section is devoted to the construction of explicit solutions to the NLSDP model \eqref{deltasch1} of the form  \eqref{numeratorwith} with $\omega \neq 0$ (standing waves) and of the form \eqref{phi0} with $\omega = 0$ (equilibrium solutions). For that purpose let us consider the following general problem 
\begin{equation}\label{genod}
\left\{
\begin{aligned}
&-A_Z\phi+ \omega \phi+f(|\phi|^2)\phi=0,\\
&\phi\in H^1(\mathbb R)\backslash \{0\},
\end{aligned}
\right.
\end{equation}
where $f = f(\cdot)$ is an arbitrary function satisfying
\begin{subequations}
\begin{align}
f \in C^1((0,+\infty); \R) \quad \text{with} \; \; f(0) = 0, \label{ci}\\
f'(x) < 0  \quad \text{for all} \; \, x > 0. \label{cii}
\end{align}
\end{subequations}
%
%
%\begin{itemize}
%\item[(i)] $f\in C^1([0,\infty);\mathbb{R})$ with $f'(x)<0$ for all $x>0$,
%\item[(ii)] $f(0)=0$.
%\end{itemize}
For example,  if $1 < p < \infty$, $\lambda_1\leq 0$ and $\lambda_2<0$, with $|\lambda_1| + |\lambda_2| \neq 0$, then the function
\begin{equation}
\label{genf12}
f(x) = \lambda_1 x^{(p-1)/2} + \lambda_2 x^{p-1},
\end{equation}
satisfies conditions \eqref{ci} and \eqref{cii}.

We recall that $\phi\in H^1(\mathbb R)\backslash \{0\}$ is a solution in the distributional sense  for \eqref{genod} if for every $\chi\in H^1(\mathbb R)$ we have (see \eqref{quadra}-\eqref{quadra2})
\begin{equation}\label{distri}
\begin{aligned}
0=&\langle A_Z\phi- \omega \phi-f(|\phi|^2)\phi, \chi \rangle= F_Z[\phi,\chi]-\langle \omega \phi+f(|\phi|^2)\phi, \chi \rangle\\
=& \Re\Big [\int_{-\infty}^{+\infty} \phi'(x)\overline{\chi'(x)}dx-Z \phi(0)\overline{\chi(0)} -\omega \int_{-\infty}^{+\infty}\phi(x)\overline{\chi(x)}dx -\int_{-\infty}^{+\infty}f(|\phi|^2(x))\phi(x)\overline{\chi(x)}dx \Big].
\end{aligned}
\end{equation}

Next lemma establishes the principal properties of the solutions to equation \eqref{genod} when $\phi\in H^1(\mathbb{R})$. This result will be useful throughout the variational analysis in section 4.1.

\begin{lemma}\label{regul} Let $\omega \in\mathbb{R}$, $Z\in\mathbb{R}\setminus\{0\}$ and let $f$ satisfy \eqref{ci} and \eqref{cii}. Suppose $\phi\in H^1(\mathbb{R})$ is a distributional solution to \eqref{genod}. Then $\phi$ satisfies the following properties:
\begin{subequations}
\begin{align}
&\phi\in C^j(\mathbb{R} \setminus \{0\})\cap C(\mathbb{R}),\qquad j=1,2,\label{reg}\\
&\phi''(x)+\omega\phi(x)+f(|\phi(x)|^2)\phi(x)=0,\qquad \text{for all } \;\; x\neq 0. \label{eqdifx}\\
&\phi'(0+)-\phi'(0-)=-Z\phi(0),\label{condsal}\\
&\phi'(x),\phi(x)\rightarrow 0, \qquad \qquad \;\; \text{if} \;\; |x|\rightarrow\infty,\label{complimitado}\\
&|\phi'(x)|^2+\omega|\phi(x)|^2+g(|\phi(x)|^2)=0, \qquad \text{for all} \;\; x\neq 0, \label{eqdifxre}
\end{align}
\end{subequations}
where,   
\[
g(x)=\int_0^xf(s)ds.
\]
\end{lemma}

\begin{proof}
The proof of this lemma follows the ideas of the proof of Lemma 3.1 in \cite{FuJe08}. By convenience of the reader we give a sketch of it. Indeed, properties \eqref{reg} and \eqref{complimitado} are proved by a standard boostrap argument, namely, for all $\xi\in C_0^{\infty}(\mathbb{R}\setminus\{0\})$ the function $\xi\phi$ satisfies, for every $\chi\in H^1(\mathbb R)$,
\[
\langle (\xi\phi)''+\omega(\xi\phi),\chi\rangle = \xi''\phi+2\xi'\phi'-\xi f(|\phi|^2)\phi, \chi\rangle,
\]
because $\phi$ is a solution to \eqref{distri} with $\chi$ substituted by $\chi \xi$. Thus, we obtain the following equality in the distributional sense (in $H^{-1}(\mathbb R)$),
\begin{equation}\label{distri2}
 (\xi\phi)''+\omega(\xi\phi) = \xi''\phi+2\xi'\phi'-\xi f(|\phi|^2)\phi.
\end{equation}
Since the right hand side of the previous identity is in $L^2(\mathbb{R})$, then $\xi\phi\in H^2(\mathbb{R})$, that is, $\phi\in H^2(\mathbb{R}\setminus \{0\})\cap C^1(\mathbb{R}\setminus\{0\})$. Thus, using equality \eqref{distri2} once again we immediately obtain \eqref{reg} and \eqref{complimitado}. Moreover, since $\phi\in C^2(\mathbb{R}\setminus \{0\})$ and $\xi\in C_0^{\infty}(\mathbb{R}\setminus\{0\})$, dense in $L^2(\mathbb{R})$, we obtain that equality\eqref{distri2} is true pointwise for all $x\neq 0$ and it implies \eqref{eqdifx}.

Next, since $\phi$ satisfies \eqref{distri} for every $\chi\in H^1(\mathbb R)$ (continuous in $x=0$) we obtain, after integration by parts  in \eqref{distri} and after using \eqref{eqdifx}, that
\[
0=\Re \Big( \big(\phi'(0+)-\phi'(0-)+Z\phi(0) \big) \overline{\chi(0}) \Big).
\]
Thus, we obtain relation \eqref{condsal}. Now, from \eqref{eqdifx} we deduce that 
\begin{equation}\label{distri4}
\begin{aligned}
&\frac12 \frac{d}{dx}\Big(|\phi'(x)|^2+\omega|\phi(x)|^2+g(|\phi(x)|^2)\Big)=\\
&\Re(\phi''(x)\overline{\phi'(x)})+\omega\Re(\phi(x)\overline{\phi'(x)})+f(|\phi(x)|^2)\Re(\phi(x)\overline{\phi'(x)})=0.
\end{aligned}
\end{equation}
Hence, integrating the previous equation with respect to the variable $x$ and using \eqref{complimitado} we obtain \eqref{eqdifxre}.
\end{proof}

The following general lemmata will be useful later on.

\begin{lemma}
\label{positive}
Let $p>1$, $\omega,\lambda_1,\lambda_2\in\mathbb{R}$ and $Z\in\mathbb{R}\setminus\{0\}$. Let $\phi$ be a non-trivial solution to \eqref{reg} - \eqref{eqdifxre}. Then $\phi(x)\neq 0$ for all $x\in\mathbb{R}$ and $|\phi|>0$. Notice that  $-\phi$ is a solution to \eqref{eqdifx} as well.
\end{lemma}

\begin{remark} In Lemma \ref{seis} below we show the existence of a unique positive (negative) solution for \eqref{eqdifx}  satisfying \eqref{reg}, \eqref{condsal} and \eqref{complimitado}. By the analysis in section 3.1, that positive solution must be determined by the profile in \eqref{numeratorwith} in the case $Z>0$.

\end{remark}
\begin{proof}[Proof of Lemma \ref{positive}]
We argue by contradiction. If there exists $x_0>0$ such that $\phi(x_0)=0$, then from \eqref{eqdifxre} we obtain that $\phi'(x_0)=0$. Now, since $\phi$ satisfies \eqref{eqdifx} for $x>0$, and $\phi(x_0)=\phi'(x_0)=0$, then from the Cauchy principle we conclude that $\phi(x)=0$ for all $x>0$. Hence, from \eqref{reg}   
and \eqref{condsal}, we infer that $\phi(0)=0$ and $\phi'(0+)=\phi'(0-)=0$, respectively. Then $\phi\in C^2(\mathbb R)$ and satisfies \eqref{eqdifx} for all $x\in \mathbb R$ with $\phi(0)=\phi'(0)=0$, therefore $\phi\equiv 0$.
 Similar arguments work if $x_0< 0$. Now, if $\phi(0)=0$ then from the identity \eqref{eqdifxre} we get that $\phi'(0+)=\phi'(0-)=0$. Hence, applying the same arguments discussed above we conclude that $\phi(x)=0$ for all $x\in\mathbb{R}$. This ends the proof of the lemma.
\end{proof}

\begin{lemma}
\label{carac}  
Let $p>1$, $\omega,\lambda_1,\lambda_2\in\mathbb{R}$ and $Z\in\mathbb{R}\setminus\{0\}$. Let $\phi$ be a non-trivial solution to \eqref{reg} - \eqref{eqdifxre}. Then we have either one of the following:
\begin{itemize}
\item[(i)] $\Im(\phi(x))=0$ for all $x\in\mathbb{R}$; or,
\item[(ii)] there exists $c\in\mathbb R$ such that $\Re(\phi(x))=c \, \Im(\phi(x))$ for all $x\in\mathbb{R}$. 
\end{itemize}
\end{lemma}
\begin{proof} If we define $\Re(\phi(x))=u(x)$ and $\Im(\phi(x))=v(x)$ then we obtain that $u$ and $v$ satisfy 
\begin{equation}\label{system}
\left\{
\begin{aligned}
&u''(x)+\omega u(x)+f(|\phi(x)|^2)u(x)=0,\\
&v''(x)+\omega v(x)+f(|\phi(x)|^2)v(x)=0,\\
\end{aligned}
\right.
\end{equation}
for all $x\neq 0$. Therefore, $(u'(x)v(x)-u(x)v'(x))'=0$ for all $x\neq 0$. Since $u,u',v,v'$ tend to zero at infinity, then we get
\begin{equation}\label{ideb}
u'(x)v(x)=u(x)v'(x),\hspace{0.4cm}\text{ for all }\hspace{0.5cm}x\in\mathbb{R}\setminus\{0\}. 
\end{equation}
If $v(x_0)=0$ for $x_0\neq 0$ then \eqref{ideb} and Lemma \ref{regul} yield $v'(x_0)=0$. Thus, from the second equation in \eqref{system} and Lemma \ref{regul} we conclude that $v(x)=0$ for all $x\in\mathbb{R}$. Similarly, if $v(0)=0$ then from \eqref{ideb} we get that $v'(0+)=0$ and hence $v(x)=0$ for all $x\in\mathbb{R}$, which implies that (i) holds. On the other hand, if $v(x)\neq 0$ for all $x\in\mathbb{R}$ then
\[
\frac{d}{dx}\left(\frac{u(x)}{v(x)}\right)=\frac{u'(x)v(x)-u(x)v'(x)}{v(x)^2}=0,\hspace{0.3cm}\text{for all}\hspace{0.3cm}x\in\mathbb{R}\setminus\{0\}.
\]
Since $u,v$ are continuous functions and $v(0)\neq 0$, the previous identity implies that there exists a constant $c\in\mathbb{R}$ such that $u(x)=cv(x)$ for all $x\in\mathbb{R}$. That is, (ii) holds. 
\end{proof}

According to Lemma \ref{regul}, in order to determine explicit solutions of the equation \eqref{genod} for all $x\in\mathbb{R}$, it is necessary to compute explicit regular solutions to equation \eqref{eqdifx} ($x\neq 0$) such that condition \eqref{complimitado} is satisfied. Hence, we use the latter to construct a new solution satisfying the jump condition \eqref{condsal}. This will be the general strategy to construct both standing waves and equilibrium solutions with point defects for $Z \neq 0$.

\subsection{Standing waves with $\omega \neq 0$}

In this subsection we construct explicit standing wave solutions to \eqref{deltasch1} with $\omega \neq 0$ and $Z\neq 0$. For that purpose, let us first compute positive solutions $\phi$ to \eqref{genod} with $f$ of the form \eqref{genf12} and with $Z = 0$.
%First, we compute explicit solutions to \eqref{genod} when $f(x)=\lambda_1x^{\frac{p-1}{2}}+\lambda_2x^{p-1}$, $Z=0$ and $\phi$ is a positive function. 
Under these assumptions we obtain that $\phi\equiv \phi_\omega$ satisfies the nonlinear elliptic equation
\begin{equation}\label{ordendoisz0}
\phi''(x)+\omega\phi(x)+\lambda_1\phi^{p}(x)+\lambda_2\phi^{2p-1}(x)=0,\hspace{1,5cm}x\in\mathbb{R}.
\end{equation}
By a quadrature procedure and by considering the boundary condition for the profile $\phi(x)\to 0$ as $|x|\to \infty$, we obtain
\[
[\phi']^2+\omega\phi^2+2\alpha\phi^{p+1}+\beta\phi^{2p}=0.
\]
Thus, for $\lambda_1\leq 0$, $\lambda_2<0$ we define $\alpha := \lambda_1/(p+1)$ and $\beta := \lambda_2/p$. In order to obtain an explicit solution to equation \eqref{ordendoisz0}, we will assume $\alpha^2-\omega\beta<0$ and $\phi>0$. Upon substitution of $y=\phi^{p-1}$ into the previous formula we deduce that
\begin{equation}\label{inttrans}
\frac{1}{(p-1)\sqrt{-\beta s}}\bigintsss \frac{du}{(u-x_0)\sqrt{u^2+1}}=x,
\end{equation}
whereupon $s=-\beta^{-2}(\alpha^2-\omega\beta)>0$, $x_0=\alpha\beta^{-1}s^{-1/2}$ and $u=s^{-1/2}y+x_0$.
Next, for a positive constant $c > 0$ we obtain
\begin{equation*}\label{inttransit}
\sqrt{1+x_0^2}\bigintssss \frac{du}{(u-x_0)\sqrt{1+u^2}}=\log\left[\frac{c(u-x_0)\sqrt{1+x_0^2}}{2(1+x_0u)+2\sqrt{1+x_0^2}\sqrt{1+u^2}}\right].
\end{equation*}
Thus, if we define
\[
H(u) :=\frac{u-x_0}{1+x_0u+\sqrt{1+x_0^2}\sqrt{1+u^2}},
\]
then it is possible to verify that $H:[x_0,\infty)\rightarrow[0,\sqrt{1+x_0^2}-x_0)$ is a diffeomorphism with inverse given by
\[
H^{-1}(u)=\frac{x_0u^2-2u-x_0}{u^2+2x_0u-1}=x_0-\frac{2(1+x_0^2)}{u-u^{-1}+2x_0}.
\]
Henceforth, for $c=\frac{2}{\sqrt{1+x_0^2}}$ and since $\sqrt{-\beta s}\sqrt{1+x_0^2}=\sqrt{-\omega}$, we rewrite \eqref{inttrans} as 
\[
H(u)=e^{(p-1)\sqrt{-\omega}x}.
\]
Since $u=H^{-1}\left(e^{(p-1)\sqrt{-\omega}x}\right)$, $u=s^{-1/2}y+x_0$ and $y=\phi^{p-1}$, we obtain that the profile function
\begin{equation}\label{matria}
\phi(x)=\left[\frac{\omega}{\sqrt{\omega\beta-\alpha^2}\sinh((p-1)\sqrt{-\omega}x)-\alpha}\right]^{\frac{1}{p-1}},
\end{equation}
is a positive solution to equation \eqref{ordendoisz0} defined for  $x\in (-\infty,l_\omega)$, where 
\[
l_\omega :=\frac{1}{(p-1)\sqrt{-\omega}}\sinh^{-1}\left(\frac{\alpha}{\sqrt{\omega\beta-\alpha^2}}\right)<0.
\]
In addition, this solution satisfies
\[
\lim_{x\to -\infty}\phi(x)=0,\hspace{1.0cm}\text{and}\hspace{1.0cm}\lim_{x\to l_\omega}\phi(x)=+\infty,
\]
that is, $\phi$ decays to zero at $-\infty$ and blows up at $x=l_\omega$. We also observe that the solution $\phi$ will be well defined for $\lambda_1\leq 0$, $\lambda_2<0$, provided that
\[
-\frac{p\lambda_1^2}{(p+1)^2\lambda_2}<-\omega.
\]
In particular $\omega$ must be negative. If $\lambda_2\geq 0$ then the function in \eqref{matria} is not well defined.

Next, we use the non-continuous profile in \eqref{matria} to construct a continuous profile satisfying the relations in Lemma \ref{regul} when $f$ is of the form \eqref{genf12} and $Z\neq 0$. Thus, we notice that for values of $d$ (to be specified later) satisfying  $0<l_\omega+d$, we define for $x\leq 0$ the half-profile $\psi(x)\equiv \phi(x-d)$ and consider the even-profile
\begin{equation*}
\phi_1(x)=\left\{
\begin{aligned}
&\psi(x),\quad x\leq 0,\\
&\psi(-x),\quad x>0.\\
\end{aligned}
\right.
\end{equation*}
In other words, 
\begin{equation}\label{errsest}
\phi_1(x)=\phi(-|x|-d),\quad x\in \mathbb R,
\end{equation}
which satisfies $\phi_1\in H^1(\mathbb R)$ and all the properties of Lemma \ref{regul}, except for the jump condition \eqref{condsal}. In order to ensure this condition we need to restrict some parameters. Indeed,  since $\phi_1$ is an even function, condition \eqref{condsal} can be rewritten as
\[
\phi_1'(0-)=\frac{Z}{2}\phi_1(0),\hspace{0.5cm}\text{or equivalently,}\hspace{0.5cm}\phi'(-d)=\frac{Z}{2}\phi(-d).
\]
Hence, \eqref{matria} yields that the shift $d$ needs to satisfy the relation
\begin{equation}\label{translat}
R_1(d)=\frac{Z}{2\sqrt{-\omega}},
\end{equation}  
where $R_1:(-l_\omega,\infty)\rightarrow (1,\infty)$ is given by
\[
R_1(d)=\frac{\sqrt{\omega\beta-\alpha^2}\cosh((p-1)\sqrt{-\omega}d)}{\sqrt{\omega\beta-\alpha^2}\sinh((p-1)\sqrt{-\omega}d)+\alpha}. 
\]
It is not difficult to verify that $R_1$ is a decreasing diffeomorphism from the interval $(-l_\omega,\infty)$ to the interval $(1,\infty)$. In particular, we conclude that 
\[
Z>0, \qquad\text{and}\;\; -\omega<\frac{Z^2}{4},
\]
and furthermore, that
\begin{equation}\label{ese}
d=R_1^{-1}\left(\frac{Z}{2\sqrt{-\omega}}\right).
\end{equation}
Lastly, from \eqref{matria}, \eqref{errsest} and \eqref{ese} we conclude that the even-profile $\phi_{1}=\phi_{\omega}$ given by
\begin{equation}\label{numeratorwithZ1}
\phi_{\omega}(x)=\left[\frac{\alpha}{-\omega}+\frac{\sqrt{\omega\beta -\alpha^2}}{-\omega}\sinh\left((p-1)\sqrt{-\omega}\left(|x|+R_1^{-1}\left(\frac{Z}{2\sqrt{-\omega}}\right)\right)\right)\right]^{-\frac{1}{p-1}},
\end{equation}
is a solution to \eqref{peak}. Whence, we have proved the following existence result of even-peak standing wave solutions to equation \eqref{deltasch1}.

\begin{theorem}
\label{solution} 
Let $p>1$, $\lambda_1\leq 0$, $\lambda_2 <0$  and  $Z>0$ in equation \eqref{deltasch1}.  Then for all values of $\omega<0$ satisfying 
\[
-\frac{p\lambda_1^2}{(p+1)^2\lambda_2}<-\omega<\frac{Z^2}{4}
\]
the familiy of standing wave solutions, $u(x,t)=e^{-i\omega t}\phi_{\omega}$, with $\phi_{\omega}$ given by formula \eqref{numeratorwithZ1} are solutions to the double power nonlinear Schr\"odinger equation \eqref{deltasch1}.  
\end{theorem} 

Figure \ref{figprofiles}  below shows the profile of $\phi_{\omega}$ in \eqref{numeratorwithZ1} in  the case of repulsive intraspecies quintic/cubic nonlinearities ($p=3$ and $\lambda_1 = \lambda_2 =-1$), with a delta point interaction of strength $Z = 2$. The frequency of oscillation is taken as $\omega = -0.25$ (and hence satisfying the assumption of Theorem \ref{solution}). The peaked shape of the profile is due to the effect of the $\delta$-trapping potential at the origin. 

\begin{figure}[ht]
\begin{center}
\includegraphics[scale=0.65, clip=true]{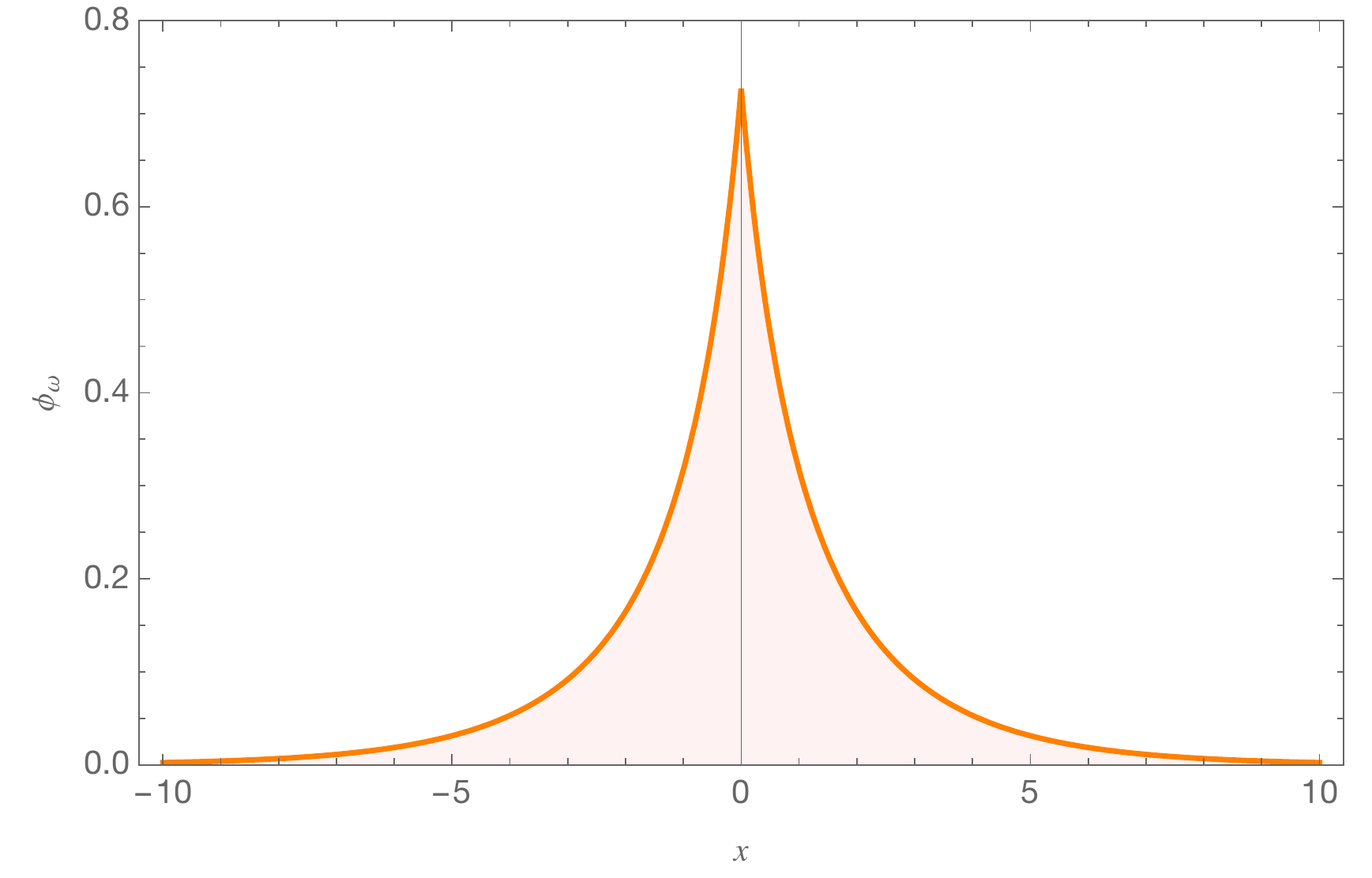}
\end{center}
\caption{Profile function $\phi_\omega = \phi_\omega(x)$ defined in \eqref{numeratorwithZ1} for parameter values $\omega = -0.25$, $Z = 2$, in the case of a cubic/quintic ($p=3$), doubly repulsive ($\lambda_1 = \lambda_2 = -1$) nonlinearity (color online).}\label{figprofiles}
\end{figure}

Figure \ref{figstanding} below shows the time evolution of the standing wave of the form $u(x,t) = e^{-i\omega t} \phi_\omega(x)$ as solution to the NLSDP model \eqref{deltasch1}
 for the same set of parameter values: $\lambda_1 = \lambda_2 = -1$ (repulsive interactions), $p=3$ (cubic/quintic nonlinearities) and $Z = 2$, with oscillation frequency $\omega = -0.25$.

\begin{figure}[h]
\begin{center}
\subfigure[$\Re (e^{-i \omega t} \phi_\omega(x))$]{\label{panela}\includegraphics[scale=.75, clip=true]{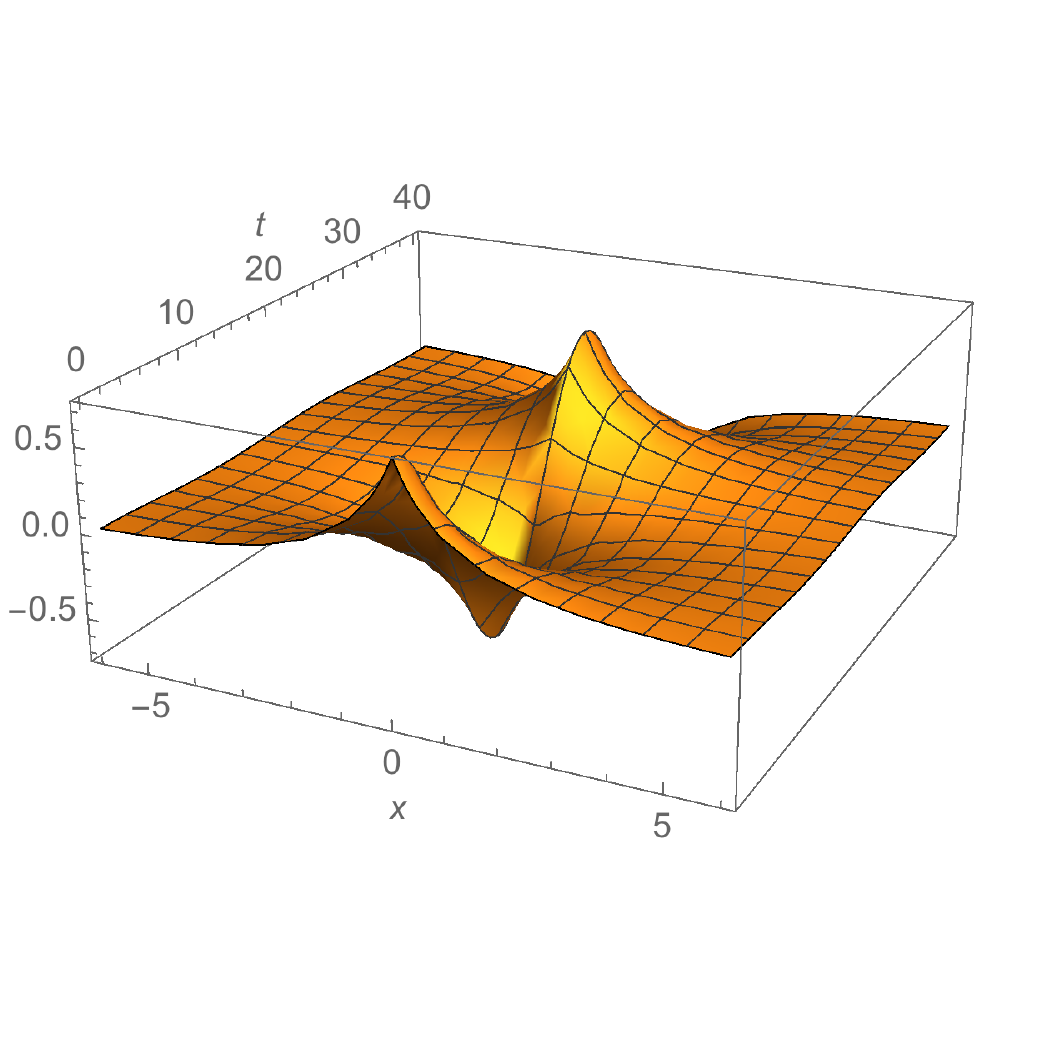}}
\subfigure[Contour plot]{\label{panelb}\includegraphics[scale=.65, clip=true]{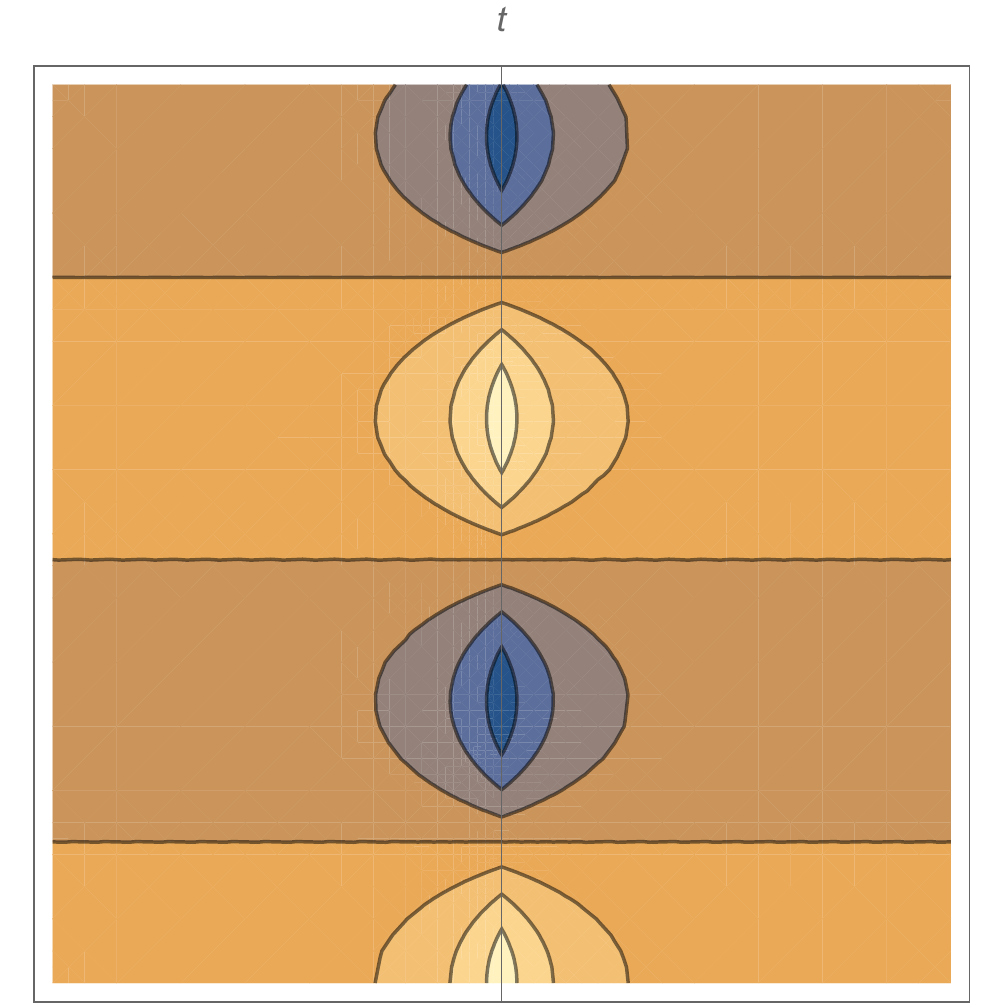}}
\end{center}
\caption{Time evolution of the standing wave solution $u(x,t) = e^{-i\omega t} \phi_\omega(x)$ where the profile function $\phi_\omega$ is defined in \eqref{numeratorwithZ1} with oscillation frequency $\omega = -0.25$, delta point interaction of strength $Z = 2$ and in the case of a quintic/cubic ($p=3$), doubly repulsive ($\lambda_1 = \lambda_2 = -1$) nonlinearity. The left panel \ref{panela} depicts $\Re u(x,t)$ for the parameter values under consideration. The right panel \ref{panelb} shows the contour plot of $\Re u(x,t)$ versus $x$ (horizontal) and $t$ (vertical) axes (color online).}\label{figstanding}
\end{figure}

\begin{remark}\label{phaseplane} 
Notice that equation \eqref{eqdifx} is a Hamiltonian system. Thus, the qualitative behavior of its solutions can be analyzed from the study of the level sets of its Hamiltonian function,
\[
\Phi_\omega(x,y)=y^2+\omega x^2+g(x^2), \qquad x,y\in\mathbb{R}.
\]

If $\omega < 0$ and $f$ is of the form \eqref{genf12}, $f(x) = \lambda_1 x^{(p-1)/2} + \lambda_2 x^{p-1}$, where $p$, $\lambda_1$, $\lambda_2$ and $Z$ satisfy the hypotheses of Theorem \ref{solution}, it is not difficult to verify that $(0,0)$ is the only equilibrium point of the function $\Phi_\omega$ and that it is a saddle. Hence, the ordinary differential equation in \eqref{eqdifx} has $(0,0)$ as the unique equilibrium point and it is hyperbolic (see Figure \ref{figHam}). Thus, any solution leaving $(0,0)$ (i.e., vanishing at $x = -\infty$) lies on the unstable manifold at the origin. Any solution that arrives at $(0,0)$ (i.e., vanishing at $x = +\infty$) lies on the stable manifold at the origin. Therefore, for a jump with given intensity $Z > 0$, the solution $\phi$ satisfying \eqref{reg} - \eqref{eqdifxre} leaves the origin along one branch of the unstable manifold and jumps to the stable manifold on the same half plane with same sign for $\phi$, returning to $(0,0)$ along it. That is, the only solution is either $\phi_\omega$ or $-\phi_\omega$, where $\phi_\omega = \phi_\omega(\cdot)$ is given by the profile function \eqref{numeratorwithZ1}. Since for $\omega < 0$ the origin is a hyperbolic point, the trajectory leaves and returns to the origin exponentially fast, just like the asymptotic behavior as $|x| \to \infty$ of the profile function $\phi_\omega = \phi_\omega(x)$ in \eqref{numeratorwithZ1}. 
 
 \begin{figure}[h]
\begin{center}
\includegraphics[scale=0.65, clip=true]{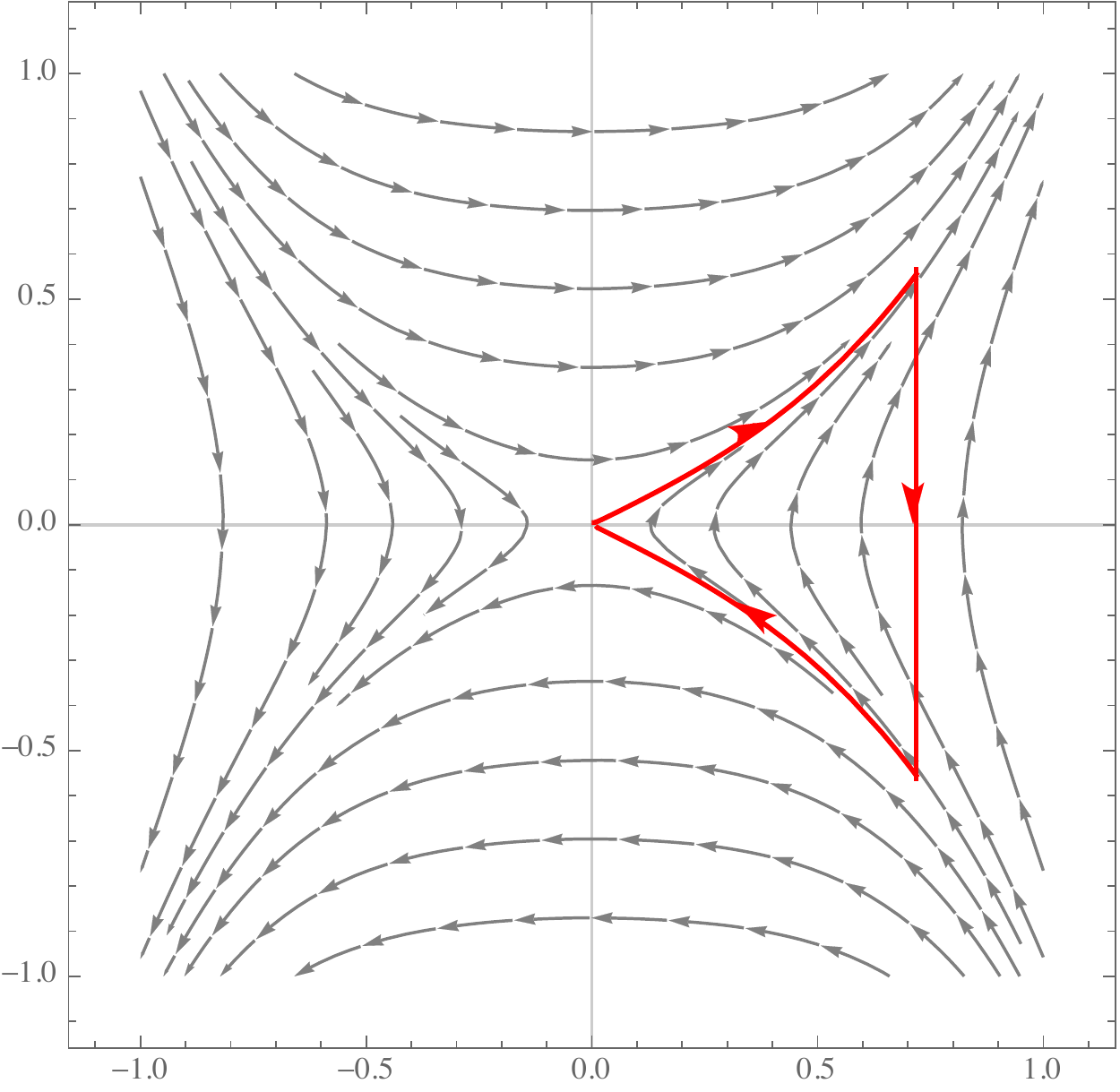}
\end{center}
\caption{Dynamics in the $(\phi, \phi')$-plane for equation \eqref{eqdifx} with $f(x) = - x (1+x^2)$, that is, for parameter values $\lambda_1 = \lambda_2 = -1$ and $\omega = -0.25$, in the case of a quintic/cubic nonlinearity with $p = 3$. The solutions to \eqref{reg} - \eqref{eqdifxre} leave the origin along the unstable manifold (trajectory in red; color online) for positive values of $\phi$, makes a jump in $\phi'$ of intensity $Z > 0$ to the stable manifold on the same side of the phase plane and returns to the origin along it.}\label{figHam}
\end{figure}
 
 \end{remark}

\subsection{Equilibrium solutions of rational profile}
\label{equilibriumsol} 

In this section we construct explicit equilibrium solutions ($\omega=0$) to the nonlinear Schr\"odinger equation \eqref{deltasch1} with $\lambda_1<0$ and $\lambda_2<0$. As before, we start by considering the case with $Z = 0$ and then we proceed to construct a solution to the equation with $Z \neq 0$ out of it. Upon substitution of $\omega=0$ and $Z=0$ into \eqref{peak} we obtain that $\phi\equiv \phi_0$ satisfies the nonlinear elliptic equation 
\begin{equation}\label{edoomega0}
\phi''(x)+\lambda_1\phi^{p}(x)+\lambda_2\phi^{2p-1}(x)=0,\hspace{1,5cm}x\in\mathbb{R}.
\end{equation}
Once again, one uses a quadrature procedure and applies the boundary condition for the profile $\phi(x)\to 0$ as $|x|\to \infty$ to arrive at
\begin{equation}
\label{ord1z0omega0}
[\phi']^2+2\alpha\phi^{p+1}+\beta\phi^{2p}=0,
\end{equation}
where $\alpha=\lambda_1/(p+1)<0$ and $\beta=\lambda_2/p<0$. In order to obtain an explicit solution to equation \eqref{edoomega0},  let us assume that $\phi > 0$. Upon substitution of $y=\phi^{p-1}$ into the previous formula, we deduce that
\begin{equation}
\label{inttransomega0}
\frac{1}{(p-1)\sqrt{-\beta }x_0}\bigintssss \frac{du}{u^{3/2}\sqrt{1+u}}=x,
\end{equation}
where $x_0=2\alpha/\beta$ and $u=y/x_0$. Since
\begin{equation*}\label{inttransitomega0}
\bigintssss \frac{du}{u^{3/2}\sqrt{1+u}}=-2\left(\frac{1}{u}+1\right)^{\frac{1}{2}} + c,
\end{equation*}
for any positive constant of integration $c > 0$, then it is not hard to verify that $\phi$ has the even-rational profile 
\begin{equation}\label{matriaomega0}
\phi(x)=\left[\frac{-2p(p+1)\lambda_1}{p(p-1)^2\lambda_1^2x^2+(p+1)^2\lambda_2}\right]^{\frac{1}{p-1}},
\end{equation} 
and it is a positive solution of  equation \eqref{ord1z0omega0} defined at least for $x\in (-\infty,l_0)$, where 
\begin{equation}\label{blowupomega0}
l_0=\frac{\sqrt{|\lambda_2|}(p+1)}{\sqrt{p}(p-1)\lambda_1}< 0,
\end{equation}
and satisfies
\[
\lim_{x\to -\infty}\phi(x)=0\qquad \text{and,} \qquad \lim_{x\to l_0^-}\phi(x)=+\infty,
\]
so that $\phi$ decays to zero at $-\infty$ and blows up at $x=l_0$. 

Next, we proceed to build  solutions to equation \eqref{peak} when $Z \neq 0$ and $\omega=0$. Thus, for $\lambda_1<0$, $\lambda_2< 0$ and $d$ satisfying  $0<l_0+d$, the function 
\begin{equation}\label{errsestomega0}
\phi_0(x) :=\phi(-|x|-d)
\end{equation}
satisfies $\phi_0\in H^1(\mathbb R)$ and  the relations of Lemma \ref{regul} except for the jump condition \eqref{condsal}. In order to ensure this condition we need to restrict some parameters. Indeed,  since $\phi_0$ is an even function, the condition \eqref{condsal} can be rewritten as
\[
\phi_0'(0+)=-\frac{Z}{2}\phi_0(0),\hspace{0.5cm}\text{or equivalently,}\hspace{0.5cm}\phi'(-d)=-\frac{Z}{2}\phi(-d).
\] 
Hence, from \eqref{matriaomega0} we obtain that $R_2(d)=Z/4$, where $R_2:(-l_0,\infty)\rightarrow (0,\infty)$ is given by
\begin{equation}\label{Erreomega0}
R_2(d)=\frac{d}{(p-1)(d^2-l_0^2)}. 
\end{equation}
It is possible to verify that $R_2$ is a decreasing diffeomorphism from the interval $(-l_0,\infty)$ to the interval $(0,\infty)$. In particular, from expression \eqref{Erreomega0} we conclude that $Z > 0$ and 
\begin{equation}\label{eseomega0}
d=R_2^{-1}\left(\frac{Z}{4}\right).
\end{equation}

Finally, from \eqref{matriaomega0}, \eqref{errsestomega0} and \eqref{eseomega0} we conclude that the function $\phi_0$ given by
\begin{equation}
\label{nutorwithZ1}
\phi_0(x)=\left[\frac{-2p(p+1)\lambda_1}{p(p-1)^2\lambda_1^2\left(|x|+R_2^{-1}\left(\frac{Z}{4}\right)\right)^2+(p+1)^2\lambda_2} \right]^{\frac{1}{p-1}}
\end{equation}
is a solution to \eqref{peak} provided that $\lambda_1<0$, $\lambda_2< 0$ and $Z > 0$. Therefore, we have established the following existence result of peak standing wave solutions to the NLSDP model \eqref{deltasch1}.

\begin{theorem}[equilibrium solutions]
\label{eqsolution} 
 Let $1 < p < 5$ and set $\omega=0$ in \eqref{peak}. For $\lambda_1<0$, $\lambda_2 < 0$ and any $Z>0$ the family of positive even-rational profiles defined by \eqref{nutorwithZ1} are equilibrium solutions to the nonlinear Schr\"odinger equation \eqref{deltasch1}. 
\end{theorem}

\begin{remark}
\label{rempl5}
Even though the construction of the profile function \eqref{nutorwithZ1} works for a larger regime of parameter values, we specialize the statement of the existence theorem to the case $1 < p < 5$ in view that a direct inspection of the profile yields that $\phi_0 \notin L^2(\R)$ for $p \geq 5$ and $\phi_0 \in H^1(\R)$ for $1 < p < 5$, as the reader may directly verify.
\end{remark}

Figure \ref{figequilibria} above shows the profile function $\phi_0$ defined in \eqref{nutorwithZ1} in the doubly repulsive case ($\lambda_1=\lambda_2=-1$) for $Z = 1.25$ and in the case of cubic/quintic nonlinearities, $p=3$.

\begin{figure}[h]
\begin{center}
\includegraphics[scale=0.65, clip=true]{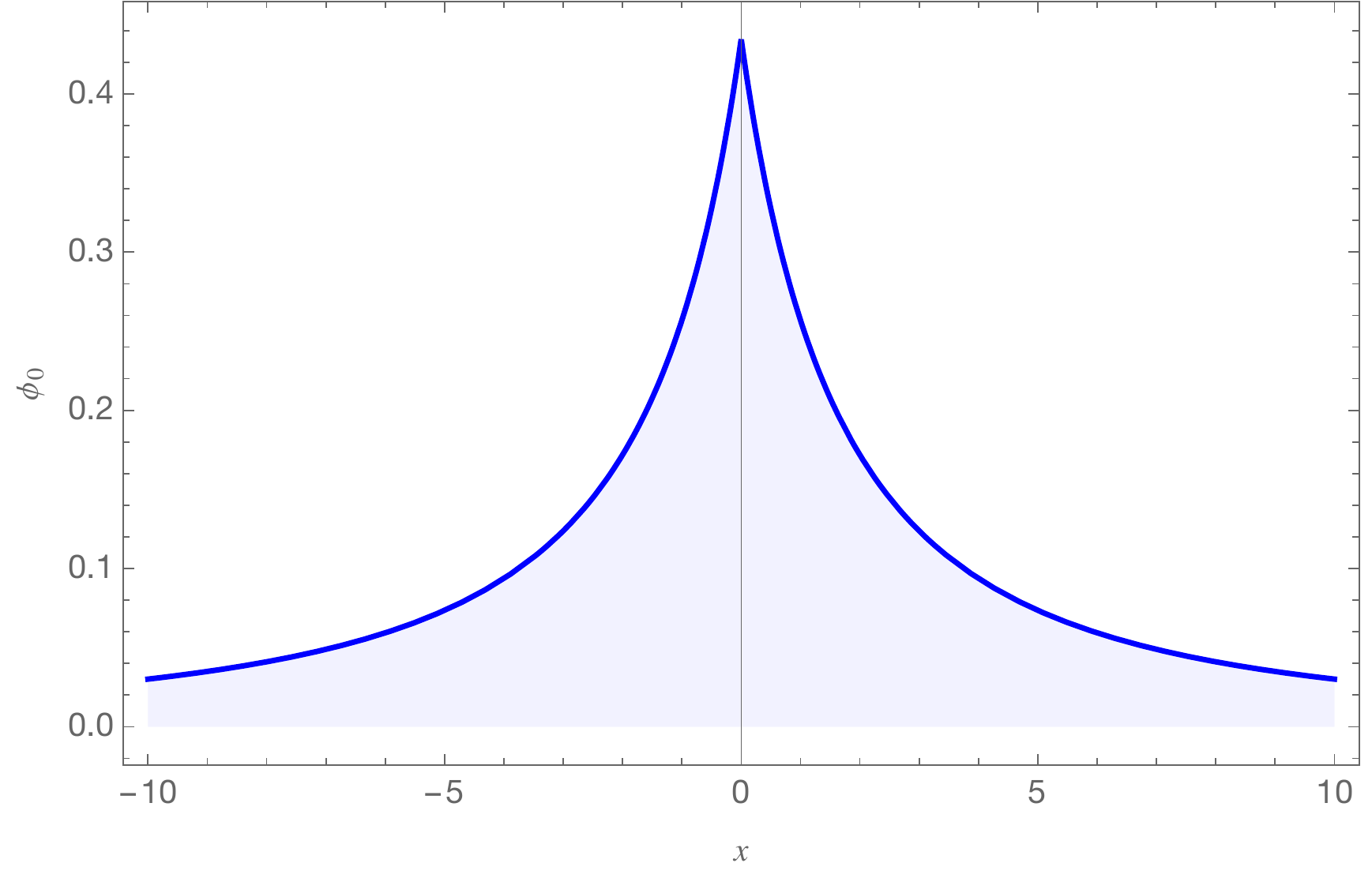}
\end{center}
\caption{Depiction of the profile function $\phi_0 = \phi_0(x)$ defined in \eqref{nutorwithZ1} for $Z = 1.25$, in the case of a cubic/quintic ($p=3$), doubly repulsive ($\lambda_1 = \lambda_2 = -1$) nonlinearity. (Color online.)}\label{figequilibria}
\end{figure}

\begin{remark}\label{phaseplaneom0} 
When $\omega = 0$, the dynamics of equation \eqref{eqdifx} in the phase plane depends on the Hamiltonian
\[
\Phi_0(x,y)=y^2+g(x^2), \qquad x,y\in\mathbb{R}.
\]

If $\omega = 0$ and $f$ is of the form \eqref{genf12}, $f(x) = \lambda_1 x^{(p-1)/2} + \lambda_2 x^{p-1}$, where $1 < p < \infty$, $\lambda_1 < 0$, $\lambda_2 \leq 0$ and any $Z > 0$, the equilibrium point $(0,0)$ in the $(\phi, \phi')$-plane is degenerate. Still, thanks to hypotheses \eqref{ci} -\eqref{cii}, it is the only equilibrium point of the system (see Figure \ref{figHamom0}). By degeneracy of the origin, the stable manifold is tangent to the center manifold $\phi' = 0$. Thus, any solution leaving $(0,0)$ (i.e., vanishing at $x = -\infty$) lies on the unstable manifold at the origin. Any solution that arrives at $(0,0)$ (i.e., vanishing at $x = +\infty$) lies on the stable manifold at the origin, also tangentially to $\phi' = 0$. Therefore, for a jump with given intensity $Z > 0$, the solution $\phi$ satisfying \eqref{reg} - \eqref{eqdifxre} leaves the origin along one branch of the unstable manifold and jumps to the stable manifold on the same half plane with same sign for $\phi$, returning to $(0,0)$ along it. That is, the only solution is either $\phi_0$ or $-\phi_0$, where $\phi_0 = \phi_0(\cdot)$ is given by the profile function \eqref{nutorwithZ1}. It is to be observed that thanks to the degeneracy of the equilibrium point in the case where $\omega = 0$, the trajectory leaves and returns to the origin algebraically fast, just like the asymptotic behavior of the profile function $\phi_0 = \phi_0(\cdot)$ in \eqref{nutorwithZ1} as $|x| \to \infty$ (this can be also appreciated from Figure \ref{figHamom0}). 
%This algebraic decay of the profile $\phi_0$ will play a role in the characterization of the set of critical points of the charge/energy functional as we shall see in Lemma \ref{lemB} below.
 \begin{figure}[h]
\begin{center}
\includegraphics[scale=0.65, clip=true]{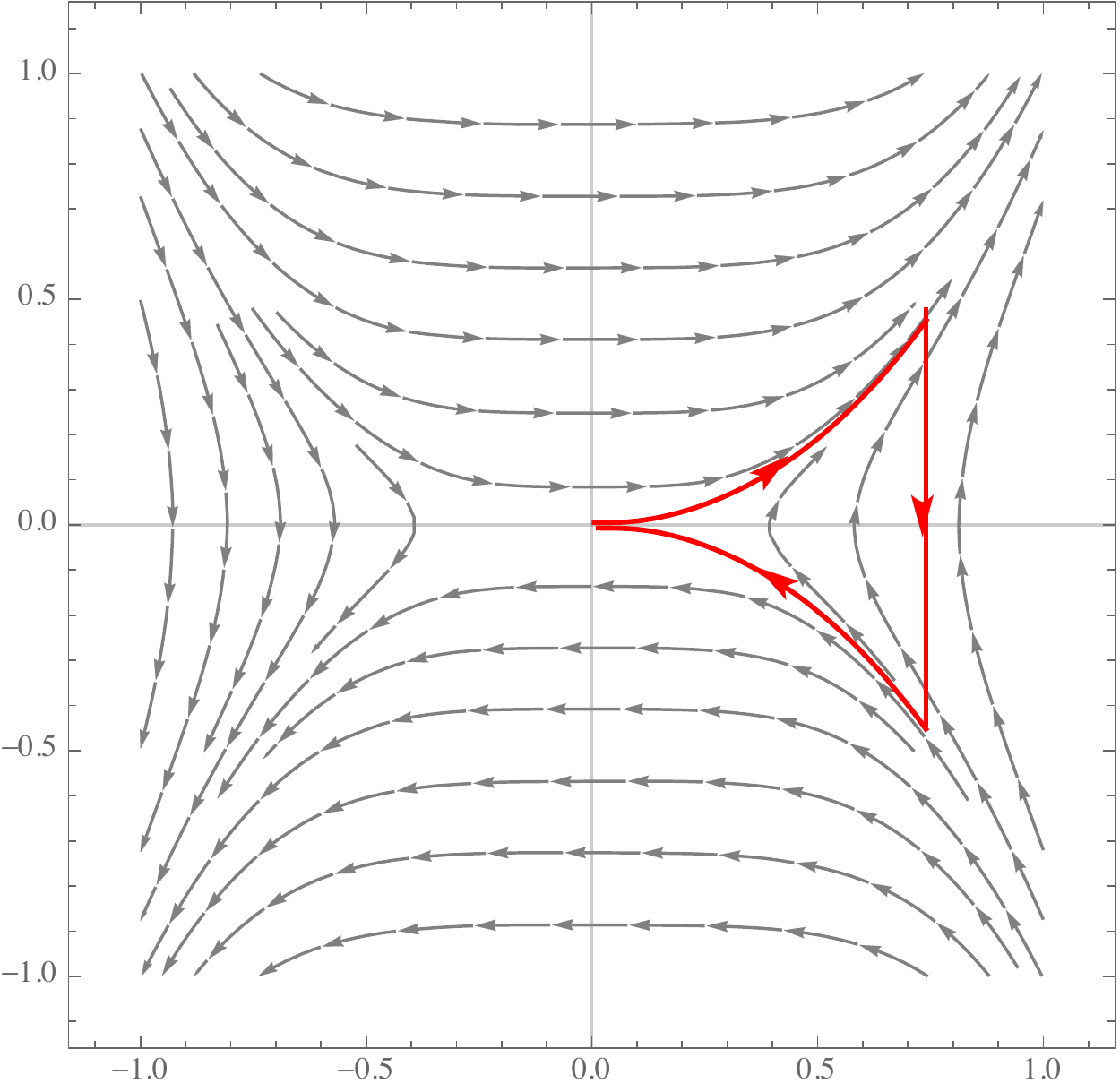}
\end{center}
\caption{Dynamics in the $(\phi, \phi')$-plane for equation \eqref{eqdifx} with $\omega = 0$ and $f(x) = - x (1+x^2)$, that is, for parameter values $\lambda_1 = \lambda_2 = -1$ in the case of a quintic/cubic nonlinearity with $p = 3$. The solutions to \eqref{reg} - \eqref{eqdifxre} leave the origin along the unstable manifold (trajectory in red; color online) for positive values of $\phi_0$, makes a jump in $\phi'_0$ of intensity $Z > 0$ to the stable manifold on the same side of the phase plane and returns to the origin along it. Both stable and unstable manifolds at the origin are tangent to the $\phi$-axis, yielding the algebraic decay of the profile \eqref{nutorwithZ1} as $|x| \to \infty$.}\label{figHamom0}
\end{figure}
 \end{remark}

\section{Stability theory}
\label{secstab}

In this section we prove Theorems \ref{main} and \ref{main2}, based on the minimization of the charge/energy functional \eqref{ge} and on the uniqueness (modulo rotations) of the even-positive profiles in \eqref{numeratorwithZ1} and \eqref{nutorwithZ1}.

\subsection{Description of the set of critical points}
Let us consider the functional $G_{\omega}:H^1(\mathbb R)\to \mathbb R$ for values $\omega \leq 0$, defined as
\begin{equation}
\label{genfunctional}
G_{\omega}(v)=\frac{1}{2}\|v_x \|^2_{L^2}-\frac{Z}{2}|v(0)|^{2}-\frac{\omega}{2}\|v\|^2_{L^2}-\frac{1}{2}\int_{-\infty}^{\infty} g(|v(x)|^2)dx,
\end{equation}
and the set of critical point associated to $G_{\omega}$ as
\[
{\cal A}_{\omega}=\{v\in H^1(\mathbb R) :  G'_{\omega}(v)=0, v\neq 0\}.
\]
Here $g = g(\cdot)$ is the antiderivative of any function $f = f(\cdot)$ satisfying assumptions \eqref{ci} - \eqref{cii}.
We notice that for  $\phi\in {\cal A}_{\omega}$ we have the relation
\[
G'_{\omega}(\phi)=A_Z\phi-\omega\phi-f(|\phi|^2)\phi
\]
in the sense that for every $\chi\in H^1(\mathbb R)$
\[
\langle G'_{\omega}(\phi), \chi\rangle =F_Z[\phi, \chi]-\langle \omega\phi+f(|\phi|^2)\phi, \chi\rangle.
\]
Observe that the operator $u\to G'_{\omega}(u)$ is continuous from $H^1(\mathbb R)$ to $H^{-1}(\mathbb R)$. So, we have that $\phi$ is a solution to equation \eqref{genod} if and only if $\phi\in H^1(\mathbb R)\setminus \{0\}$ and $G'_{\omega}(\phi)=0$ ($\phi\in {\cal A}_{\omega}$). Thus we have that \eqref{genod} is the Euler-Lagrange equation associated to the functional $G_\omega$. Moreover, for 
$\phi\in {\cal A}_{\omega}$ it follows, by Lemma \ref{regul}, that $\phi$ satisfies \eqref{reg} - \eqref{eqdifxre}.

The following Lemmata show the non-existence of non-trivial solutions and the uniqueness of positive solutions for \eqref{genod}.

\begin{lemma}
\label{nule}
Let $1 < p < \infty$, $Z>0$ and let $\omega \in \R$ be such that $\omega+\frac{Z^2}{4}\leq 0$. Then  the set ${\cal A}_{\omega}$ is empty. 
\end{lemma}
\begin{proof}
We argue by contradiction. If there exists $h\in H^1(\mathbb R)\setminus \{0\}$ satisfying the stationary problem $G'_{\omega}(h)=0$, then 
\[
0=\frac{d}{ds}G(sh)\Big |_{s=1}=\langle G'(h),h\rangle=\langle A_Z h,h\rangle
-\omega\|h\|^2_{L^2}-\int_{-\infty}^{\infty} f(|h(x)|^2)|h(x)|^2dx.
\]
Now, since for $Z>0$ 
\[
\langle A_Z h,h\rangle\geq -\frac{Z^2}{4}\|h\|^2_{L^2}
\]
for all $h\in H^1(\mathbb R)$, we then obtain
\begin{equation}\label{inew}
\begin{aligned}
0&=\|h_x \|^2_{L^2}-Z|h(0)|^{2}-\omega\|h\|^2_{L^2}-\int_{-\infty}^{\infty} f(|h(x)|^2)|h(x)|^2dx\\
&\geq -(Z^2/4+\omega)\|h\|^2_{L^2}-\int_{-\infty}^{\infty} f(|h(x)|^2)|h(x)|^2dx\\
&\geq-\int_{-\infty}^{\infty} f(|h(x)|^2)|h(x)|^2dx>0,
\end{aligned}
\end{equation}
where we have used the fact that $-f \geq 0$ (an increasing continuous function) on the interval $(0,\infty)$ with $f(0)=0$ and that $h$ is a non-trivial continuous function. The inequalities in \eqref{inew} provide a contradiction. Hence ${\cal A}_\omega=\varnothing$, as claimed.
\end{proof}

\begin{lemma}\label{cinco} 
Let $1 < p < \infty$ and $Z\in\mathbb{R} \setminus \{0\}$. If $\omega>0$ then ${\cal A}_{\omega}=\varnothing$.
\end{lemma}
\begin{proof}
We proceed by contradiction. If there exists $\phi\in H^1(\mathbb R)\setminus \{0\}$ satisfying \eqref{reg} - \eqref{eqdifxre} then, by the mean value theorem for integrals, we can rewrite \eqref{eqdifxre} in the following form:
\begin{equation}\label{meanv}
|\phi'(x)|^2+[\omega+ f(r(x))]|\phi(x)|^2=0, \qquad \text{for all } \; x\neq 0, 
\end{equation}
where $r(x)\in(0, |\phi(x)|^2)$. Now, since $\lim_{|x|\to\infty}|\phi(x)|^2=0$ and $f$ is a continuous function at $x=0$ with $f(0)=0$, we obtain
\[
\lim_{|x|\to\infty}f(r(x))=0.
\]
Thus, there exists $L>0$ such that $0<f(r(x))+\frac{\omega}{2}$,  for all $x$ with $|x|>L$.  Therefore, 
\[
0<|\phi'(x)|^2+[\omega+ f(r(x))]|\phi(x)|^2, \quad \text{for all } \; x>L,
\]
a contradiction with \eqref{meanv}.
\end{proof}

\begin{lemma}
\label{Z<0} 
Suppose $f = f(\cdot)$ is any function satisfying assumptions \eqref{ci} - \eqref{cii}. Let $\omega\in \mathbb R$. Then, we have that ${\cal A}_{\omega}=\varnothing$.
\end{lemma}
\begin{proof}
We proceed by contradiction. By Lemmas \ref{positive} and \ref{carac} we can consider a real valued $\phi$ solution to \eqref{reg} - \eqref{eqdifxre} such that $\phi>0$. Then, multiply \eqref{eqdifx} by $\phi$, integrate by parts and use \eqref{condsal} to obtain
\begin{equation}
\label{1a}
\frac{Z}{2}\phi^2(0)-\frac12 \int_{-\infty}^{+\infty} (\phi')^2dx +\frac{\omega}{2} \int_{-\infty}^{+\infty} \phi^2dx + \frac{1}{2} \int_{-\infty}^{+\infty} f(\phi^2) \phi^2 \, dx = 0.
%
%\frac{\lambda_1}{2} \int_{-\infty}^{+\infty} \phi^{p+1}dx +\frac{\lambda_2}{2} \int_{-\infty}^{+\infty} \phi^{2p}dx =0.
\end{equation}
%\begin{equation}
%\label{1a}
%\frac{Z}{2}\phi^2(0)-\frac12 \int_{-\infty}^{+\infty} (\phi')^2dx +\frac{\omega}{2} \int_{-\infty}^{+\infty} \phi^2dx +\frac{\lambda_1}{2} \int_{-\infty}^{+\infty} \phi^{p+1}dx +\frac{\lambda_2}{2} \int_{-\infty}^{+\infty} \phi^{2p}dx =0.
%\end{equation}
Next, multiplying \eqref{eqdifx} by $x\phi'$ and integrating by parts yields
\begin{equation}
\label{1b}
-\frac12 \int_{-\infty}^{+\infty} (\phi')^2dx -\frac{\omega}{2} \int_{-\infty}^{+\infty} \phi^2dx - \frac{1}{2} \int_{-\infty}^{+\infty} g(\phi^2) \, dx = 0.
%\frac{\lambda_1}{p+1} \int_{-\infty}^{+\infty} \phi^{p+1}dx -\frac{\lambda_2}{2p} \int_{-\infty}^{+\infty} \phi^{2p}dx =0.
\end{equation}
%\begin{equation}
%\label{1b}
%-\frac12 \int_{-\infty}^{+\infty} (\phi')^2dx -\frac{\omega}{2} \int_{-\infty}^{+\infty} \phi^2dx -\frac{\lambda_1}{p+1} \int_{-\infty}^{+\infty} \phi^{p+1}dx -\frac{\lambda_2}{2p} \int_{-\infty}^{+\infty} \phi^{2p}dx =0.
%\end{equation}
Thus, from \eqref{1a} and \eqref{1b} we obtain 
\begin{equation}
\label{1c}
\frac{Z}{2}\phi^2(0)-\int_{-\infty}^{+\infty} (\phi')^2dx = - \frac{1}{2} \int_{-\infty}^{+\infty} f(\phi^2) \phi^2 \, dx + \frac{1}{2} \int_{-\infty}^{+\infty} g(\phi^2) \, dx.  
%
%-\lambda_1\frac{p-1}{2(p+1)}  \int_{-\infty}^{+\infty} \phi^{p+1}dx -\lambda_2\frac{2(p-1)}{4p} \int_{-\infty}^{+\infty} \phi^{2p}dx,
\end{equation}
%\[
%\frac{Z}{2}\phi^2(0)-\int_{-\infty}^{+\infty} (\phi')^2dx= -\lambda_1\frac{p-1}{2(p+1)}  \int_{-\infty}^{+\infty} \phi^{p+1}dx -\lambda_2\frac{2(p-1)}{4p} \int_{-\infty}^{+\infty} \phi^{2p}dx,
%\]
%it which obviously can not happen by the conditions on the parameters  $p, \lambda_1, \lambda_2, Z$. 
Like in the proof of Lemma \ref{cinco}, let us write $g(\phi(x)^2) = f(r(x)) \phi(x)^2$ for all $x \in \R$ where $r(x) \in (0, \phi(x)^2)$. Since $f$ is monotone decreasing we have $- f(r(x)) < - f(\phi(x)^2)$. This yields $- g(\phi(x)^2) < - f(\phi(x)^2) \phi(x)^2$. Integrating this inequality we obtain,
\[
- \int_{-\infty}^{+\infty} f(\phi^2) \phi^2 \, dx + \int_{-\infty}^{+\infty} g(\phi^2) \, dx > 0,
\]
that is, the right hand side of \eqref{1c} is positive, a contradiction with the sign of the left hand side when $Z < 0$. This finishes the proof.
\end{proof}

Let us now specialize the function $f$ to the particular form \eqref{genf12}, where the parameters $\lambda_1<0$, $\lambda_2 <0$, $Z > 0$ and $\omega <0$ are such that standing wave profiles do exist (see Theorem \ref{solution}). The following lemma characterizes the set of non-trivial critical points in that case.

\begin{lemma}
\label{seis}
Let $p>1$, $\lambda_1<0$, $\lambda_2<0$, $Z>0$ and $\omega$ such that $-\frac{p\lambda_1^2}{(p+1)^2\lambda_2}<-\omega<\frac{Z^2}{4}$. Consider $f(x) = \lambda_1 x^{(p-1)/2} + \lambda_2 x^{p-1}$. Then ${\cal A}_{\omega}=\{e^{i\theta}\phi_{\omega}: \theta\in\mathbb{R}\}$ where $\phi_{\omega}$ denotes the standing wave profile given in \eqref{numeratorwithZ1}.
\end{lemma}
\begin{proof} 
It is clear that for all $\theta\in\mathbb{R}$, $e^{i\theta}\phi_{\omega}\in{\cal A}_{\omega}$. Conversely, if $g\in{\cal A}_{\omega}$, then $g$ satisfies \eqref{reg} - \eqref{eqdifxre} and by Lemma \ref{positive} $|g|>0$. We will  show that
there exist $\theta\in\mathbb{R}$ such that $g(x)=e^{i\theta}\phi_\omega(x)$ for all $x\in\mathbb{R}$. 

Initially, we show that $\phi_\omega\in D(A_Z)$ is the unique positive solution for \eqref{genod}. Indeed,  from Lemma \ref{positive} is sufficient to consider $v\in H^1(\mathbb R)$ satisfying $v(x)>0$ for all $x\in \mathbb R$ and the properties \eqref{reg} - \eqref{eqdifxre}. We consider the following polynomial $P:(0, +\infty)\to \mathbb R$ defined by
\begin{equation}\label{rela0}
P(c)=\frac12 \frac{Z^2}{4} c^2 + F(c),\quad F(c)=\int_0^c \omega t + f(t^2)t \, dt,
\end{equation}
and the following initial value problem (IVP) on $(0,+\infty)$,
\begin{equation}\label{ivp}
\left\{
\begin{aligned}
-\psi''(x)&=H(\psi(x)),\quad x> 0,\\
\psi(0)&=c_0,\\
\psi'(0)&=-\frac{Z}{2}c_0,
\end{aligned}
\right.
\end{equation}
where $H(\psi)=\omega \psi +f(|\psi|^2)\psi$ and $c_0$ is the unique positive root of $P$ (to be determined below). Thus, since $H$ is a locally Lipschitz function around zero, we have that the IVP \eqref{ivp} has a unique  (positive) solution and it is given exactly by $\phi_\omega$. Indeed, since $v$ satisfies $-v''=H(v)$ we obtain by integration for any $R>0$ that
\[
\int_0^R-v'(x)v''(x)dx=\int_0^R H(v(x))v'(x)dx=\int_0^R F(v(x))v'(x)dx=F(v(R))-F(v(0+)).
\]
Thus, for all $R>0$, $\frac{1}{2}(v'(0+))^2-\frac{1}{2}(v'(R))^2-F(v(R))+F(v(0+))=0$.
Then for $R\to +\infty$ we obtain
\begin{equation}\label{rela1}
\frac{1}{2}(v'(0+))^2+F(v(0+))=0.
\end{equation}
Similarly, we have 
\begin{equation}\label{rela2}
\frac{1}{2}(v'(0-))^2+F(v(0-))=0.
\end{equation}
Next, since $v$ is continuous in $x=0$ we get from \eqref{rela1} and \eqref{rela2} that
$|v'(0+)|=|v'(0-)|$. Now we suppose that $v'(0+)=v'(0-)$. Hence, from Lemma \ref{regul}, there holds $v(0)=0$. Now, we divide our analysis into two steps:
\begin{itemize}
\item[(i)] if  $v'(0+)=v'(0-)=0$ then $v'(0)=0$ and the IVP \eqref{ivp} with $c_0=0$ has a unique solution, $v\equiv 0$. This is a contradiction with $v(x)>0$ for all $x \in \R$; 
\item[(ii)] if  $v'(0+)=v'(0-)\neq 0$ then there is $x_0\in \mathbb R$ with $v(x_0)<0$ for $x_0$ close to zero, which cannot happen.
\end{itemize}

From the analysis above, we necessarily have $v'(0+)=-v'(0-)$ and, from the jump condition, we obtain $v'(0+)=-\frac{Z}{2}v(0)$. Now, let us define $v(0+)=c>0$. From \eqref{rela1} we then obtain
\begin{equation}\label{rela3}
\frac{1}{2} \frac{Z^2}{4}v^2(0+)+F(v(0+))=0.
\end{equation}
Thus from \eqref{rela0} it follows that $P(v(0+))=0$. Next we determine the existence of a unique zero for $P$ on $(0,+\infty)$. Since
\[
P(c)=\frac{1}{2} \Big(\frac{Z^2}{4}+\omega \Big)c^2 + \frac{\lambda_1}{p+1}c^{p+1}+\frac{\lambda_2}{2p}c^{2p},
\]
we have $P'(0+)>0$, $\lim_{c\to +\infty} P(c)=-\infty$, and there is a unique critical point $a>0$ with $P(a)>0$. Indeed, $P'(c)=0$ if and only if $r=c^{p-1}$ satisfies the quadratic equation $\lambda_2 r^2+ \lambda_1 r +\frac{Z^2}{4}+\omega=0$. Since 
$\frac{Z^2}{4}+\omega>0$ and $\lambda_1, \lambda_2<0$, this polynomial has exactly two different real roots, with a unique positive root $r_0 > 0$. Thus $a=r_0^{\frac{1}{p-1}}$. Therefore,  there is $c_0>0$ ($c_0>a$) such that $P(c_0)=0$ and $c_0$ only depends {\it a priori} on the parameters $Z, \omega, \lambda_1,\lambda_2$.

Then, since $P(v(0+))=0$ with $v(0+)>0$ we need to have $v(0+)=c_0$. Therefore, $v$ is the unique (local) solution for the IVP \eqref{ivp}, at least for $x\in (0, a_0)$. Now, since $v\in C^j(0, +\infty)$, $j=0,1,2$ and $v(x)\to 0$ as $x\to +\infty$, it follows that $v\in L^{\infty}(0, +\infty)$. Therefore, from standard ODE arguments we can choose $a_0=+\infty$ and, consequently, the unique solution of \eqref{ivp} on $(0, +\infty)$ is positive. A similar analysis (but now on $(-\infty, 0)$) shows that $v$ is the unique solution to \eqref{ivp} on $(-\infty, 0)$.  Therefore, since $\phi_\omega$ is a continuous profile satisfying the IVP \eqref{ivp} on $(-\infty, 0)$ and $(0,+\infty)$, necessarily $v=\phi_\omega$.

Lastly, since $g\in C^2(0,+\infty)$ we can write $g(x)=e^{i\theta(x)}\rho(x)$, where $\theta, \rho\in C^2(0,+\infty)$, real-valued functions,  and $\rho>0$ ($|g|>0$). Thus, by substituting $g$ in \eqref{eqdifx} and taking real and imaginary part we obtain the system
\begin{equation}\label{system}
\left\{
\begin{aligned}
\theta''\rho+2\theta'\rho' &=0,\quad x> 0,\\
-(\theta')^2\rho+\rho''+\omega \rho +f(|\rho|^2)\rho&=0,\quad x> 0.
\end{aligned}
\right.
\end{equation}
Therefore, there is a real constant $K$ such that $\rho^2\theta'=K$ on $(0,+\infty)$. Now since $|g'|$ is bounded we get that $\rho^2(\theta')^2=\frac{K^2}{\rho^2}$ is bounded on $(0,+\infty)$. But, since $\rho(x)\to 0$ as $x\to +\infty$ we need to have $K=0$. Then $\theta(x)=\theta_0$ for all $x\in (0,+\infty)$. Thus, $g(x)=e^{i\theta_0}\rho(x)$ for all $x\in (0,+\infty)$. From second equation in \eqref{system} we obtain that $\rho$ is a positive solution for \eqref{eqdifx} and satisfying \eqref{reg}, \eqref{condsal}, \eqref{complimitado} and \eqref{eqdifxre}. Thus, by the analysis above we necessarily have that $\rho(x)=\phi_\omega(x)$ for all $x\in (0,+\infty)$. A similar analysis shows that $g(x)=e^{i\theta_1}\phi_\omega(x)$ for all $x\in (-\infty, 0)$. Finally, by continuity we obtain $e^{i\theta_1}=e^{i\theta_0}$ and hence $g(x)=e^{i\theta_0}\phi_\omega(x)$ for all $x\in \mathbb R$. This finishes the proof.
\end{proof}

A similar uniqueness result for $\omega = 0$ is obtained. We omit the proof.

\begin{lemma}
\label{lemB}
Let $1 < p < 5$. Assume that $\lambda_1 < 0$, $\lambda_2< 0$, $Z > 0$, $\omega = 0$ and consider $f(x) = \lambda_1 x^{(p-1)/2} + \lambda_2 x^{p-1}$. Then $\mathcal{A}_0 = \{e^{i \theta} \phi_0 \, : \, \theta \in \R\}$ where $\phi_0$ is the equilibrium solution defined in \eqref{nutorwithZ1}.
\end{lemma}

\subsection{Orbital stability of standing waves for $\omega\neq 0$}

This section is devoted to prove Theorem \ref{main}. Let us suppose that the parameter values satisfy $1 < p < \infty$, $Z > 0$, $\lambda_1 \leqq  0$, $\lambda_2 < 0$ and that $\omega$ is such that
\[
-\frac{p\lambda_1^2}{(p+1)^2\lambda_2}<-\omega<\frac{Z^2}{4}.
\]
Consider a function $f$ of the form \eqref{genf12}, that is,
\[
f(x) = \lambda_1 x^{(p-1)/2} + \lambda_2 x^{p-1},
\]
and the following  minimization problem associated to $G_{\omega}$ defined in \eqref{genfunctional},
\[
m(\omega)=\inf \{ G_{\omega}(v) \, : \, v \in H^1(\mathbb{R}) \},
\]
and the minimal set
\[
M(\omega)=\{ u \in H^1(\mathbb{R}) \, : \, G_{\omega}(u)=m(\omega)\}.
\]

\begin{lemma}
\label{siete}
$-\infty<m(\omega)<0$ and $M(\omega)\subset{\cal A}_{\omega}$.
\end{lemma}
\begin{proof}
We first verify that $-\infty<m(\omega)$. Indeed, let us  write
\[
G_w(v)=R(v)-\frac{\omega}{2}\|v\|^2_{L^2}-\frac{\lambda_1}{p+1}\|v\|^{p+1}_{L^{p+1}}\quad v\in H^1(\mathbb R),
\]
where $R$ is the functional defined in \eqref{deffuncR}. Then, by Lemma \ref{lemaux} we get
\[
G_w(v)\geqq R(v)\geqq \frac{Z}{2}|v(0)|^2 -C\geqq -C,
\]
for all $v\in H^1(\mathbb{R})$ and some uniform $C>0$, yielding $-\infty<m(\omega)$, as claimed. In order to show that $m(\omega)<0$, let $v(x) := sh(x) \in H^1(\R)$ with $s>0$ and where $h(x)=e^{-\frac{Z|x|}{2}}$ is the eigenfunction of the operator $A_Z$ associated to the eigenvalue $\frac{-Z^2}{4}$. Therefore
\[
G_{\omega}(v) = - \frac{s^2}{2} \Big(\frac{Z^2}{4} + \omega \Big)\|h\|^2_{L^2}-\frac{1}{2}\int_{-\infty}^{\infty} g(s^2h^2(x))dx,
\]
where $g(s^2h^2(x))=f(r(x))s^2h^2(x)$, with  $0<r(x)<s^2h^2(x)\leq s^2$. Now, since $-f$ is an increasing function, we obtain that $-f(r(x))<-f(s^2)$ and $-g(s^2h^2(x))<-f(s^2)s^2h^2(x)$.
Thus
\[
G_{\omega}(v) \leq -\frac{s^2}{2}\|h\|^2_{L^2}\Big(\frac{Z^2}{4} + \omega + f(s^2)\Big).
\]
Since $Z^2/4+\omega > 0$ and $\lim_{s\to 0^+}f(s^2)=0$ we conclude that there exists $s_0>0$ such that $Z^2/4+\omega>-f(s^2)>0$ for $0<s\leq s_0$ and so
$G_{\omega}(s_0h)<0$. Lastly, suppose $M(\omega) \neq \varnothing$ then since for $h\in M(\omega)$ we have $h\neq 0$ and  $G_{\omega}'(h)=0$, then by Lemmata \ref{regul} and \ref{seis} we obtain $M(\omega)\subset{\cal A}_{\omega}$. This finishes the proof.
\end{proof}

%The following lemma establishes compactness of a minimizing sequence.

At this point we recall the following refinement of Fatou's lemma due to Br\'ezis and Lieb \cite{BrLi83}.

\begin{lemma}[Br\'ezis-Lieb]
Let $2 \leq q < \infty$ and $\{ u_j \}$ be a bounded sequence in $L^q(\R)$ such that $u_j(x) \to u(x)$ a.e. in $x \in \R$ as $j \to \infty$. Then,
\[
\| u_j \|_{L^q}^q - \| u_j - u\|_{L^q}^q - \| u \|_{L^q}^q \rightarrow 0, \quad \text{as } \; j \to \infty.
\]
\end{lemma}

The following lemma establishes the improvement from weak to strong convergence due to convergence of the charge/energy functional.

\begin{lemma}\label{ocho} Let $h_n\in H^1(\mathbb{R})$ be such that $\lim_{n\to\infty}G_{\omega}(h_n)=m(\omega)$. Then there exists a subsequence $h_{n_j}$ and $h\in H^{1}(\mathbb{R})$ such that $\lim_{{n_j}\to\infty}h_{n_j}=h$ in $H^1(\mathbb{R})$ and $G_{\omega}(h)=m(\omega)$.

\end{lemma}
\begin{proof} First, notice that for all $v\in H^1(\mathbb{R})$
\begin{equation}\label{imyr}
I_\omega(v) := \frac{1}{2}\|v_x\|^2_{L^2}-\frac{\omega}{2}\|v\|^2_{L^2}=G_{\omega}(v) + \frac{Z}{2}|v(0)|^{2} + \frac{\lambda_1}{p+1}\|v\|^{p+1}_{L^{p+1}}+\frac{\lambda_2}{2p}\|v\|^{2p}_{L^{2p}}.
\end{equation}
Since $\omega < 0$, it follows that $I_\omega(v)$ is equivalent to $\|v\|^2_{H^1}$. From \eqref{starR} and the fact that $\lambda_1, \lambda_2<0$, we obtain
\[
\frac{1}{2}\|v_x\|^2_{L^2}-\frac{\omega}{2}\|v\|^2_{L^2}\leq G_{\omega}(v)+R(v)+C \leq 2G_{\omega}(v)+C,
\]
for some uniform $C > 0$. Hence, it is clear that if the sequence $G_{\omega}(h_n)$ converges then the sequence $h_n$ is bounded in $H^1(\mathbb{R})$. Thus, there exists a subsequence $h_{n_j}$ and  $h\in H^1(\mathbb{R})$ such that $\{h_{n_j}\}$ converges wealky to $h$ in $H^1(\mathbb{R})$. In addition, since $H^1(-1,1)$ is compactly embedded in $C[-1,1]$, we deduce that $h_{n_j}(0)\rightarrow h(0)$. Thus,
\[
m(\omega)\leq G_{\omega}(h)\leq \liminf\limits_{n_j\to\infty}G_{\omega}(h_{n_j})=m(\omega),
\] 
which implies that $h\in M(\omega)$. Now, since $h_{n_j} \rightharpoonup h$ weakly in $H^1(\R)$ we have that $h_{n_j}(x) \to h(x)$ a.e. in $x \in \R$ and also that
\begin{equation}
\label{goodh1}
\begin{aligned}
\|h_{n_j} - h \|_{L^2}^2 + \| h \|_{L^2}^2 = \|h_{n_j} \|_{L^2}^2 + o(1),\\
\|\partial_x h_{n_j} - h_x \|_{L^2}^2 + \| h_x \|_{L^2}^2 = \| \partial_x h_{n_j} \|_{L^2}^2 + o(1),
\end{aligned}
\end{equation}
as $n_j \to \infty$. Since $\| h_{n_j} \|_{H^1}$ is uniformly bounded, the Gagliardo-Nirenberg interpolation inequalities \eqref{gagliNir} imply that $\| h_{n_j} \|_{L^{p+1}}$ and $\| h_{n_j} \|_{L^{2p}}$ are uniformly bounded as well. This fact, together with $h_{n_j}(x) \to h(x)$ a.e. in $x \in \R$, allows us to apply Br\'ezis-Lieb lemma and to obtain
\begin{equation}
\label{goodhp}
\begin{aligned}
\|h_{n_j} - h \|_{L^{p+1}}^{p+1} + \| h \|_{L^{p+1}}^{p+1} = \|h_{n_j} \|_{L^{p+1}}^{p+1} + o(1),\\
\|h_{n_j} - h \|_{L^{2p}}^{2p} + \| h \|_{L^{2p}}^{2p} = \|h_{n_j} \|_{L^{2p}}^{2p} + o(1),
\end{aligned}
\end{equation}
as $n_j \to \infty$. Combine \eqref{goodh1} and \eqref{goodhp} to arrive at
\[
G_\omega(h_{n_j} - h) + G_\omega(h) = G_\omega(h_{n_j}) + o(1), \qquad \text{as } \: n_j \to \infty.
\]
From \eqref{imyr} we then have
\[
\begin{aligned}
0 \leq I_\omega(h_{n_j} - h) &\leq I_\omega(h_{n_j} - h) - \frac{\lambda_1}{p+1}\|h_{n_j} - h\|^{p+1}_{L^{p+1}}-\frac{\lambda_2}{2p}\|h_{n_j} - h\|^{2p}_{L^{2p}} \\&= G_\omega(h_{n_j} - h) + \frac{Z}{2} |h_{n_j}(0) - h(0) |^2 \\&= G_\omega(h_{n_j}) - G_\omega(h) + o(1),
\end{aligned}
\]
inasmuch as $h_{n_j}(0) \to h(0)$. This yields $h_{n_j}\rightarrow h$ in $H^1(\mathbb{R})$. This finishes the proof.
\end{proof}

\begin{remark}
\label{remconvLp}
It is to be noticed that from \eqref{imyr} we also deduce $h_{n_j} \to h$ in $L^{p+1}(\R)$ and in $L^{2p}(\R)$ in view that $I_\omega(v) \geq 0$ for all $v \in H^1(\R)$.
\end{remark}

\begin{lemma}\label{nueve}
$M(\omega)={\cal A}_{\omega}=\{e^{i\theta}\phi_{\omega}:\theta\in \mathbb{R}\}$, where $\phi_{\omega}$ denotes the standing wave profile given in \eqref{numeratorwithZ1}.
\end{lemma}
\begin{proof} From Lemmata \ref{siete} and \ref{ocho} we know that $M(\omega)\neq\varnothing$. Then there exists $h\in H^1(\mathbb{R})$ such that $G_{\omega}(h)=m(\omega)$. From Lemma \ref{siete}, $h$ is a non-trivial critical point of $G_{\omega}$, that is, $h\in {\cal A}_{\omega}$. Apply Lemma \ref{seis} to obtain that $h=e^{i\theta_0}\phi_{\omega}$ for some $\theta_0\in\mathbb{R}$. Thus, since $\phi_{\omega}\in H^1(\mathbb{R})$ and $G_{\omega}(\phi_{\omega})=m(\omega)$, then $\phi_{\omega}\in M(\omega)$. This implies that ${\cal A}_{\omega}\subset M(\omega)$. The other inclusion was proved in Lemma \ref{seis}. This finishes the proof of the lemma.    
\end{proof}

After these preparations we are now ready to prove Theorem \ref{main}.

\begin{proof}[Proof of Theorem \ref{main}]
We argue by contradiction. Suppose that the standing wave  $e^{-i\omega t}\phi_\omega$ is orbitally unstable. Then there exists $\epsilon_0>0$, a sequence $\{h_n(t)\}$ of solutions of \eqref{deltasch1} (by Theorem \ref{cazi}) and a sequence $t_n>0$, such that
\begin{subequations}
\begin{align}
&\lim_{n\rightarrow\infty}\|h_n(0)-\phi_{\omega}\|_{H^1}=0,\label{con}\\
&\inf_{\theta\in\mathbb{R}}\|h_n(t_n)-e^{i\theta}\phi_{\omega}\|_{H^1}\geq \epsilon_0.\label{div}
\end{align}
\end{subequations}
Since $G_{\omega}$ is conserved by the flow of the Schr\"odinger equation \eqref{deltasch1}, we get that $G_{\omega}(h_n(t_n))=G_{\omega}(h_n(0))$ for all $n\in\mathbb{N}$. Then \eqref{con} and continuity of $G_\omega$ yield
\[
\lim_{n\to\infty}G_{\omega}(h_n(t_n))=G_{\omega}(\phi_{\omega})=m(\omega).
\]  
Henceforth, Lemmata \ref{ocho} and \ref{nueve} imply the existence of a subsequence $h_{n_j}$ such that $h_{n_j}(t_{n_j})\to h$ with $G_\omega(h)=m(\omega)$. Then $h\in {\cal A}_{\omega}$ and $h=e^{i\theta_0}\phi_{\omega}$ for some $\theta_0 \in \R$. Therefore,
\[
\lim_{{n_j}\to\infty}h_{n_j}(t_{n_j})=e^{i\theta_0}\phi_{\omega},
\] 
in $H^1(\mathbb{R})$, which contradicts \eqref{div}. Hence, we conclude that $e^{-i\omega t}\phi_\omega$ is orbitally stable.
\end{proof}
 
\subsection{Orbital stability of equilibrium solutions}

This section is devoted to prove Theorem \ref{main2}. Let $1 < p < 5$, $Z > 0$, $\lambda_1 < 0$, $\lambda_2 < 0$, $\omega = 0$ and $f$ is of the form \eqref{genf12}. We consider the space
\[
X := \{ v \in L^{p+1}(\R) \cap L^{2p}(\R) \, : \, v_x \in L^2(\R) \}.
\]
Thus, $X$ is a reflexive Banach space with norm
\[
\|v\|_X^2 = \|v\|_{L^{p+1}}^{p+1} + \|v \|_{L^{2p}}^{2p} + \|v_x \|_{L^2}^2.
\]
Notice as well that $H^1(\R) \subset X \subset L^{p+1}(\R) \cap L^{2p}(\R)$. For $\omega=0$, $G_\omega$ in \eqref{genfunctional} coincides with the conserved quantity $E$.

First, let us observe that if we define the following $C^1$ functional on $X$, $\widetilde{R}:X\to \mathbb R$, as
\begin{equation}
\label{deffunctR}
\widetilde{R}(v) := \frac{1}{2}\|v_x \|^2_{L^2}-\frac{Z}{2}|v(0)|^{2} - \frac{\lambda_1}{p+1}\|v\|^{p+1}_{L^{p+1}} = E(v) + \frac{\lambda_2}{2p}\|v\|^{2p}_{L^{2p}}, \qquad v\in X,
\end{equation}
the $\widetilde{R}(v) \leq E(v)$ and  by the same arguments as in the proof of Lemma \ref{lemaux}, it is easy to show that there exists a uniform constant $C > 0$ such that
\begin{equation}
\label{clave2}
\frac{Z}{2} |v(0)|^2 \leq \widetilde{R}(v) + C,
\end{equation}
for all $v \in X$, where we are applying the inequality
\[
\|v\|_{L^2(-1,1)}^2  \leq \delta \|v \|_{L^{p+1}(-1,1)}^{p+1} + 2C_\delta,
\]
with $\delta = - \lambda_1 / (p+1)C_1 > 0$ in view that $\lambda_1 < 0$. The rest of the proof goes verbatim.

Henceforth, we now consider the variational problem for $E: X \to \R$
\begin{align}
m_0 &= \inf \{ E(v) \, : \, v \in X \}, \nonumber\\
M_0 &= \{ v \in X \, : \, E(v) = m_0 \}, \label{varprob0}\\
\widetilde{\mathcal{A}}_0 &= \{ v \in X \, : \, E'(v) = 0, \, v \neq 0\}. \nonumber
\end{align}
Note that, clearly, $\mathcal{A}_0 \subset \widetilde{\mathcal{A}}_0$ because $G_0(v) = E(v)$ when $\omega = 0$ and $H^1(\R) \subset X$.

\begin{lemma}
\label{lemseven}
$- \infty < m_0 < 0$ and $M_0 \subset \widetilde{\mathcal{A}}_0 = \mathcal{A}_0$.
\end{lemma}
\begin{proof}
From the definition of the functional $\widetilde{R}$ and from \eqref{clave2} we clearly have
\[
E(v) \geq  \widetilde{R}(v) \geq \frac{Z}{2} |v(0)|^2 - C \geq - C,
\]
for some uniform $C > 0$, yielding $- \infty < m_0$. The proof that $m_0 < 0$ follows exactly the same ideas as for   $m(\omega) < 0$ in Lemma \ref{siete}. Lastly, if $g \in \widetilde{\mathcal{A}}_0$ then $g$ is a critical point of the variational problem \eqref{varprob0} and it satisfies \eqref{reg}-\eqref{eqdifxre} with $\omega = 0$. Since $1 < p < 5$, by the arguments of the proof Lemma \ref{lemB} we obtain that $g \in \mathcal{A}_0$. This shows that $\widetilde{\mathcal{A}}_0 \subset \mathcal{A}_0$. This concludes the proof.
\end{proof}

\begin{lemma}
\label{lemeight}
Let $\{h_n\} \subset X$ be a minimizing sequence such that $\lim_{n \to \infty} E(h_n) = m_0$. Then there exists a subsequence $h_{n_j}$ and $h \in X$ such that $\lim_{n_j \to \infty} h_{n_j} = h$ in $X$ and $E(h) = m_0$.
\end{lemma}
\begin{proof}
First let us observe that for any $v \in X$
\[
0 \leq \frac{1}{2} \| v_x\|_{L^2}^2 - \frac{\lambda_1}{p+1} \|v \|_{L^{p+1}}^{p+1} - \frac{\lambda_2}{2p} \|v \|_{L^{2p}}^{2p} = E(v) + \frac{Z}{2}|v(0)|^2 \leq E(v) + \widetilde{R}(v) + C \leq 2E(v) + C,
\]
with uniform $C > 0$ in view of estimate \eqref{clave2}. Hence it is clear that if $E(h_n)$ converges then the sequence $h_n$ is bounded in $X$. Since the space $X$ is reflexive there exist a subsequence $h_{n_j}$ and $h \in X$ such that $h_{n_j} \rightharpoonup h$ weakly in $X$. Hence we have that , $h_{n_j} \rightharpoonup h$ weakly in $L^{p+1}(\R)$, $h_{n_j} \rightharpoonup h$ weakly in $L^{2p}(\R)$ and $\partial_x h_{n_j} \rightharpoonup \partial_x h$ weakly in $L^2(\R)$. By classical interpolation inequalities in bounded intervals (cf. Br\'ezis \cite{Brez11}, chapter 6) one can show that
\[
|h_{n_j}(x)|^2 \leq C \| h_{n_j} \|_{L^{p+1}(-1,1)}^{1-\theta}  \| \partial_x h_{n_j} \|_{L^2(-1,1)}^{\theta},
\]
for $\theta = 2/(p+3) \in (0,1)$ and a.e. in $x \in (-1,1)$, yielding $h_{n_j} \in L^2(-1,1)$ (upon integration in a bounded interval). Therefore, $h_{n_j} \in H^1(-1,1)$ and since $H^1(-1,1)$ is compactly embedded in $C[-1,1]$, we deduce that $h_{n_j}(0)\rightarrow h(0)$. Consequently, we conclude that
\[
m_0 \leq E(h) \leq \liminf_{n_j \to \infty} E(h_{n_j})= m_0,
\]
which implies that $h \in M_0$. By the same arguments of the proof of Lemma \ref{ocho} (see also Remark \ref{remconvLp}) we conclude that 
\[
\|h_{n_j}\|_{L^{p+1}} \to \|h\|_{L^{p+1}}, \quad \|h_{n_j}\|_{L^{2p}} \to \|h\|_{L^{2p}}, \quad \|\partial_x h_{n_j}\|_{L^2} \to \|\partial_x h\|_{L^2}
\]
(strongly), as $n_j \to \infty$, that is, $h_{n_j} \to h$ in $X$. This concludes the proof.
\end{proof}

\begin{lemma}
\label{lemnine}
Let $h_n \in H^1(\R)$ be a sequence such that $\lim_{n \to \infty} E(h_n) = E(\phi_0)$ and $\lim_{n \to \infty}\|h_n\|_{L^2} = \| \phi_0\|_{L^2}$. Then there exists a subsequence $h_{n_j}$ and $\theta_0 \in \R$ such that $h_{n_j} \to e^{i \theta_0} \phi_0$ in $H^1(\R)$ as $n_j \to \infty$.
\end{lemma}
\begin{proof}
Lemmata \ref{lemseven} and \ref{lemeight} imply that $M_0 = \mathcal{A}_0 = \{e^{i \theta} \phi_0 \, : \, \theta \in \R\}$ and that $m_0 = E(\phi_0)$. In addition, Lemma \ref{lemeight} guarantees the existence of a subsequence $h_{n_j} \in H^1(\R)$ and of $\theta_0 \in \R$ such that $h_{n_j} \to e^{i \theta_0} \phi_0$ in $X$ as $n_j \to \infty$.  Moreover, $h_{n_j} \rightharpoonup e^{i \theta_0} \phi_0$ weakly in $L^2(\R)$. In view that, by hypothesis, $\|h_n\|_{L^2} \to \| \phi_0\|_{L^2}$, we obtain
\[
\| e^{i \theta_0} \phi_0 \|_{L^2} \leq \liminf_{n_j \to \infty} \| h_{n_j} \|_{L^2} = \| \phi_0\|_{L^2} = \| e^{i \theta_0} \phi_0 \|_{L^2},
\]
thanks to weakly lower semicontinuity of norm. Hence, $h_{n_j} \to e^{i \theta_0} \phi_0$ in $L^2(\R)$. Together with the results of Lemma \ref{lemeight}, this yields $h_{n_j} \to e^{i \theta_0} \phi_0$ in $H^1(\R)$.
\end{proof}

Finally, we are able to prove Theorem \ref{main2}.
\begin{proof}[Proof of Theorem \ref{main2}]
By contradiction, let us assume that the equilibrium solution $\phi_0$ is orbitally unstable in the space $H^1(\R)$. Then there exist $\epsilon_0 > 0$, a sequence $(h_n(t))$ of solutions to equation \eqref{deltasch1} and a sequence $t_n > 0$ such that
\begin{subequations}
\begin{align}
&\lim_{n\rightarrow\infty}\|h_n(0)-\phi_0 \|_{H^1}=0,\label{con0}\\
&\inf_{\theta\in\mathbb{R}}\|h_n(t_n)-e^{i\theta}\phi_0\|_{H^1}\geq \epsilon_0.\label{div0}
\end{align}
\end{subequations}
Since energy and charge are conserved by the flow of \eqref{deltasch1} we have that $E(h_n(t_n)) = E(h_n(0))$ for all $n \in \mathbb{N}$. Thus, \eqref{con0} implies that
\[
\lim_{n \to \infty} E(h_n(t_n)) = E(\phi_0) = m_0,
\]
and $\|h_n(t_n)\|_{L^2}=\|h_n(0)\|_{L^2}\to \|\phi_0\|_{L^2}$. Apply Lemmata \ref{lemeight} and \ref{lemnine} to deduce the existence of a subsequence $(h_{n_j})$ and of $\theta_0 \in \R$ such that
\[
\lim_{n_j \to \infty} h_{n_j}(t_{n_j}) = e^{i \theta_0} \phi_0, \qquad \text{in } \; H^1(\R).
\]
This is a contradiction with \eqref{div0} and we conclude that $\phi_0$ is orbitally stable in $H^1(\R)$.
\end{proof}

\section*{Acknowledgements}

This research was conducted while J. Angulo Pava and  C. Hern\' andez Melo were visiting the Instituto de Investigaciones en Matem\'aticas Aplicadas y en Sistemas (IIMAS), Universidad Nacional Aut\'{o}noma de M\'{e}xico, Cd. de M\'{e}xico (M\'exico). Also, they would like to thank to IIMAS by the support and the warm stay. J. Angulo  was partially  supported by Grant CNPq/Brazil.  R. G. Plaza was partially supported by DGAPA-UNAM, program PAPIIT, grant IN100318.

%
%
%\bibliography{bibliografia}

%% Estilo de la bibliografia:

%\bibliographystyle{unsrt}
%\bibliographystyle{elsarticle-num}
%\bibliographystyle{amsplain}
%\bibliographystyle{siam}
%\bibliographystyle{amsalpha}
%%%%%%%%%%%%%%%%%%% 
%\bibliographystyle{newstyle}
%\bibliographystyle{unsrt}
%\bibliographystyle{jmb}

\end{document}